\documentclass[twoside,11pt]{article}

\usepackage{lastpage}

\usepackage{epsfig,epsf,fancybox}
\usepackage{amsmath}
\usepackage{mathrsfs}
\usepackage{amssymb}
\usepackage{graphicx}
\usepackage{xcolor}
\usepackage{multirow}
\usepackage{paralist}
\usepackage{verbatim}
\usepackage{galois}
\usepackage{algorithm}
\usepackage{algorithmic}
\usepackage{boxedminipage}
\usepackage{booktabs}
\usepackage{accents}
\usepackage{stmaryrd}
\usepackage{subfig}
\usepackage{epstopdf}
\usepackage{natbib}

\newtheorem{theorem}{Theorem}[section]

\newtheorem{lemma}[theorem]{Lemma}

\newtheorem{condition}{Condition}[section]
\newtheorem{definition}{Definition}[section]

\newtheorem{assumption}[theorem]{Assumption}

\newcommand{\BI}{\mathbb{I}}

\newcommand{\Prob}{\textnormal{Prob}}

\newcommand{\x}{\mathbf x}
\newcommand{\y}{\mathbf y}
\newcommand{\s}{\mathbf s}
\newcommand{\sa}{\mathbf a}
\newcommand{\g}{\mathbf g}

\newcommand{\z}{\mathbf z}
\newcommand{\w}{\mathbf w}
\newcommand{\p}{\mathbf p}
\newcommand{\argmin}{\mathop{\rm argmin}}

\newcommand{\KCal}{\mathcal{K}}
\newcommand{\OCal}{\mathcal{O}}

\newcommand{\SCal}{\mathcal{S}}
\newcommand{\SucCal}{\mathcal{SC}}

\newcommand{\br}{\mathbb{R}}

\newcommand{\ba}{\begin{array}}
	\newcommand{\ea}{\end{array}}

\definecolor{DSgray}{cmyk}{0,1,0,0}

\begin{document}

\title{Accelerating Adaptive Cubic Regularization of Newton's Method via Random Sampling}


\author{
	Xi CHEN
	\thanks{Stern School of Business, New York University, New York, NY 10012, USA. Email: xchen3@stern.nyu.edu.} \and
	Bo JIANG
	\thanks{Research Institute for Interdisciplinary Sciences, School of Information Management and Engineering, Shanghai University of Finance and Economics, Shanghai 200433, China. Email: isyebojiang@gmail.com. } \and
	Tianyi LIN
	\thanks{Department of Electrial Engineering and Computer Science, UC Berkeley, Berkeley, CA 94720, USA. Email: darren\_lin@berkeley.edu.} \and
	Shuzhong ZHANG
	\thanks{Department of Industrial and Systems Engineering, University of Minnesota, Minneapolis, MN 55455, USA; email: zhangs@umn.edu.}}


\maketitle

\begin{abstract}
In this paper, we consider an unconstrained optimization model where the objective is a sum of a large number of possibly nonconvex functions, though overall the objective is assumed to be smooth and convex. Our bid to solving such model uses the framework of cubic regularization of Newton's method. As well known, the crux in cubic regularization is its utilization of the Hessian information, which may be computationally expensive for large-scale problems. To tackle this, we resort to approximating the Hessian matrix via sub-sampling. In particular, we propose to compute an approximated Hessian matrix by either \textit{uniformly}\/ or \textit{non-uniformly}\/ sub-sampling the components of the objective. Based upon such sampling strategy, we develop accelerated adaptive cubic regularization approaches and provide theoretical guarantees on global iteration complexity of $\OCal(\epsilon^{-1/3})$ with high probability, which matches that of the original accelerated cubic regularization methods \cite{Jiang-2020-Unified} using the \textit{full}\/ Hessian information. Interestingly, we also show that in the worst case scenario our algorithm still achieves an $\OCal(\epsilon^{-5/6}\log(\epsilon^{-1}))$ iteration complexity bound. The proof techniques are new to our knowledge and can be of independent interets. Experimental results on the regularized logistic regression problems demonstrate a clear effect of acceleration on several real data sets.
\end{abstract}

\vspace{0.1cm}


\vspace{0.1cm}


\section{Introduction}

In this paper, we consider the following {\it finite-sum}
convex optimization problem:
\begin{equation}\label{Prob:main}
f^* := \min_{\x\in\br^d} \ f(\x) = \min_{\x\in\br^d}  \frac{1}{n} \sum_{i=1}^n f_i(\x),
\end{equation}
where
$f: \br^d\rightarrow\br$ is \textit{smooth} and \textit{convex}, while each component function $f_i: \br^d\rightarrow\br$ is \textit{smooth} but possibly \textit{nonconvex}. In addition, we assume $f^*>-\infty$. A variety of machine learning and statistics applications can be cast into problem~\eqref{Prob:main} where $f_i$ is interpreted as the loss of the $i$-th observation, e.g., \citet{Friedman-2001-Elements, Sra-2012-Optimization, Kulis-2013-Metric, Bottou-2018-Optimization, Goodfellow-2016-Deep}. An important special case of problem~\eqref{Prob:main} is
\begin{equation}\label{Prob:special}
\min_{\x\in\br^d} \ f(\x) := \min_{\x\in\br^d} \ \left[\frac{1}{n} \sum_{i=1}^n f_i(\sa_i^\top\x)\right],
\end{equation}
where $ f_i: \br\rightarrow\br$ and $\sa_i$ is the $i$-th observation. The formulation in Eq.~\eqref{Prob:special} finds a wide range of applications. A typical example is the (regularized) maximum likelihood estimation for generalized linear models, which includes regularized least squares and regularized logistic regression. We refer the interested readers to Section~\ref{sec:examples} for more applications in form of Eq.~\eqref{Prob:main} and Eq.~\eqref{Prob:special}.

Up till now, much of the efforts devoted to solving problem~\eqref{Prob:main}  has been on developing stochastic first-order approach~\citep{Shalev-2016-Sdca, Allen-2016-Improved}, due primarily to its simplicity nature in both theoretical analysis and practical implementation. However, stochastic gradient type algorithms are known to be sensitive to the conditioning of the problem and the parameters to be tuned in the algorithm~\citep{Xu-2016-SubNewton}. On the contrary, second-order optimization methods~\citep{Luenberger-1984-Linear} have been shown to be generally robust~\citep{Roosta-2019-SubNewton, Xu-2016-SubNewton} and less sensitive to the parameter choices~\citep{Berahas-2020-Investigation, Xu-2017-Empirical}. A downside, however, is that the second-order type algorithms are more likely to prone to higher computational costs for large-scale problems, by nature of requiring the second-order information (viz. Hessian matrix). To alleviate this, one effective approach is the so-called sub-sampled second-order methods that approximate Hessian matrix via random sampling scheme~\citep{Drineas-2018-Lectures}.

Recent trends in the optimization community tend to improve an existing method along two possible directions. The first direction of improvement is \emph{acceleration}.~\citet{Nesterov-1983-Method, Nesterov-2004-Introductory} pioneered the study of accelerated gradient-based algorithms for convex optimization. For stochastic convex optimization, \citet{Lan-2012-Optimal} developed an accelerated stochastic gradient-based algorithm. Since then, various accelerated stochastic first-order methods have been proposed (see, e.g., \citet{Shalev-2013-Accelerated, Frostig-2015-Regularizing, Ghadimi-2016-Accelerated, Allen-2017-Katyusha, Jain-2018-Accelerating, Allen-2018-Natasha}). Despite its popularity and simplicity, {the stochastic first-order approach may perform poorly for ill-conditioned instances~\citep{Roosta-2019-SubNewton} and can be sensitive to certain algorithmic parameters such as the choices of stepsizes~\citep{Berahas-2020-Investigation}. 
In contrast}, there are limited results~\citep{Song-2019-Inexact, Ghadimi-2017-Second, Ye-2020-Nesterov} on accelerated stochastic second-order approaches. The second direction of improvement is to investigate \emph{adaptive} optimization algorithms without ensuring the problem parameters such as the first and the second order Lipschitz constants. In view of implementation, it is desirable to design algorithms that adaptively adjust these parameters since they are usually unknown {\it a priori}. A typical example is adaptive gradient method (e.g.\ AdaGrad~\citep{Duchi-11-Adaptive}), which is popular in the machine learning community.

However, such improvements -- though highly desirable due to their relevance in machine learning -- are largely lacking in the context of stochastic or sub-sampling second-order algorithms. When the objective function $f$ is non-convex, sub-sampling adaptive cubic regularized Newton's methods~\citep{Kohler-2017-SubSample, Xu-2019-Newton} are capable of reaching a second-order critical point within an iteration bound of $\OCal(\epsilon^{-3/2})$. However, we are unaware of any existing accelerated sub-sampling second-order methods that are fully independent of problem parameters while maintaining superior convergence rate.
Recall that~\citet{Nesterov-2008-Accelerating} proposed an accelerated cubic regularized Newton's method with provable overall iteration complexity of $\OCal(\epsilon^{-1/3})$ for convex optimization.

Therefore, a natural question raises:
\begin{quote}
\textit{Can one develop an adaptive and accelerated sub-sampling cubic regularized method with an iteration complexity of $\OCal(\epsilon^{-1/3})$?}
\end{quote}
In this paper, we provide an affirmative answer to the above question. {In particular, by modifying the algorithm in our previous work~\citet{Jiang-2020-Unified}, we manage to develop a novel sub-sampled cubic regularization method that is adaptive and accelerated. The advantages of the proposed approach inherited from that in~\citet{Jiang-2020-Unified} include: the algorithms are  fully adaptive without requiring any problem parameters, and the cubic regularized sub-problem in the algorithms is allowed to be solved inexactly (see Condition \ref{Cond:Approx_Subprob_SAARC}) with some easy-to-satisfy approximation conditions similar to that in \citet{Birgin-2017-Worst} and~\citet{Jiang-2020-Unified}. In contrast with the algorithms in \citet{Jiang-2020-Unified}, we use the sub-sampled Hessian rather than the full Hessian in the cubic sub-problem to reduce the per-iteration computational cost, and the sub-sampled size gradually increases from a very small initial set, leading to a significant computational savings at the beginning steps of the algorithms.	Moreover, we show that our proposed algorithm has the global convergence rate of $\OCal(\epsilon^{-1/3})$ with high probability (Theorem \ref{Theorem:SAARC-Iteration-Complexity}), which matches its deterministic counterparts \citep{Jiang-2020-Unified}, requiring the availability of the \textit{full}\/ Hessian information. Although the issue of inexact Hessian has also been discussed in \citet{Jiang-2020-Unified}, the proposed Hessian approximation is based on the finite differences of the gradient, which is { more expensive
when $n$ and $d$ are large as shown in the numerical result section}. In terms of the worst-case (i.e., when the error of the sub-sampled Hessian can not be controlled)  performance of our algorithm, we show that it has a guarantee of $\OCal(\epsilon^{-5/6}\log(\epsilon^{-1}))$ iteration bound (Theorem \ref{Thm:W-AARC-Main}).
 }  
{It is worth mentioning that the sub-sampled strategy is only adopted in approximating the Hessian matrix, while the true gradient is counted exactly. The merit of our method is particularly clear when both $n$ and $d$ are large (see Figure \ref{fig:result-low-time}). Another advantage of counting the full gradient is that our algorithm has a worst-case performance guarantee (i.e., Theorem \ref{Thm:W-AARC-Main}) in addition to the standard high probability result.}

\subsection{Examples}\label{sec:examples}
In this subsection, we provide a few examples in the form of Eq.~\eqref{Prob:main} and Eq.~\eqref{Prob:special} arising from applications of machine learning. Examples for convex component functions are well known, e.g.\ the regularized least squares problem
\begin{equation*}
\min_{\x\in \br^d} \ f(\x) = \frac{1}{n}\sum\limits_{i=1}^n \left[\left( {\sa}_i^\top \x - {R}({\sa_i})\right)^2 + \lambda\| \x\|^2\right],
\end{equation*}
and the regularized logistic regression,
\begin{equation*}
\min_{\x\in \br^d} \ f(\x) = \frac{1}{n}\sum\limits_{i=1}^n \left[\ln \left( 1+\exp\left(-{R}({\sa_i}) \cdot {\sa}_i^\top \x\right) \right) + \lambda\| \x\|^2\right],
\end{equation*}
where $\sa_i \in\br^d$ and ${R}({\sa_i})$ denote the feature 
and response of the $i$-th data point respectively. To be more specific, we have ${R}({\sa_i}) \in \br$ for the least squares loss, and ${R}({\sa_i}) \in \{-1,+1\}$ for logistic regression. The parameter $\lambda>0$ is known as the regularization parameter.

Below we shall provide some examples where certain components in the finite sum may be nonconvex. Consider for instance the \textit{nonconvex support vector machine}  \citep{Mason-2000-Boosting, Wang-2017-Stochastic}, where the objective function takes the form of
\begin{equation*}
\min_{\x\in\br^d} \ f(\x) := \frac{1}{n}\sum_{i=1}^n \left[1-\tanh\left({R}({\sa_i})\cdot\sa_i^\top\x\right) + \lambda\left\|\x\right\|^2\right] ,
\end{equation*}
which is an instance of
\eqref{Prob:main} with
\begin{equation*}
f_i(\x) = 1-\tanh\left({R}({\sa_i})\cdot\sa_i^\top\x\right) + \lambda\left\|\x\right\|^2.
\end{equation*}
Indeed, for some choice of $\lambda>0$, the objective is convex but a few component functions may be nonconvex.

Another  example comes from \textit{principal component analysis (PCA)}. Consider a set of $n$ data vectors $\sa_1, \ldots, \sa_n$ in $\br^d$ and the normalized co-variance matrix $A=\frac{1}{n}\sum_{j=1}^n \sa_j\sa_j^\top$, PCA  aims to find the leading principal component.~\citet{Garber-2015-Fast} proposed a new efficient optimization for PCA by reducing the problem to solving a small number of convex optimization problems, where a critical subroutine in the method is to solve
\begin{equation*}
\min_{\x\in\br^d} \ \frac{1}{2}\x^\top\left(\mu \BI - A\right)\x + b^\top\x = \min_{\x\in\br^d} \frac{1}{n}  \sum_{j=1}^n \left[\frac{1}{2}\x^\top\left(\mu \BI - \sa_j \sa_j^{\top} \right)\x + b^\top\x \right],
\end{equation*}
where $\mu$ is larger than or equal to the maximum eigenvalue of $A$.  
Although the above formulation is convex optimization, component functions in the above optimization problem may be nonconvex.

\subsection{Related Works}
The literature on the acceleration of second-order or higher-order methods for convex optimization is somewhat limited as compared to its first-order counterpart.~\citet{Nesterov-2008-Accelerating} improved the overall iteration complexity for convex optimization from $O(\epsilon^{-1/2})$ to $O(\epsilon^{-1/3})$ by means of the so-called cubic regularized Newton's method, and further accelerated it to $O(\epsilon^{-1/(p+1)})$ \citep{Nesterov-2019-Implementable} by utilizing up to $p$-th order derivative information.~\citet{Monteiro-2012-Iteration} and~\citet{Monteiro-2013-Accelerated} proposed the Newton proximal extragradient method (A-HPE) and its acceleration, which achieved an improved iteration complexity of $O(\epsilon^{-2/7})$. Recently,~\citet{Arjevani-2019-Oracle} showed that $O(\epsilon^{-2/7})$ is actually a lower bound for the second-order methods to solve convex optimization, and thus A-HPE method is an optimal second-order method. Motivated by Monteiro and Svaiter's work, three groups of researchers independently proposed and analyzed some optimal high-order methods achieving the iteration complexity of $O(\epsilon^{-2/(3p+1)})$~\citep{Gasnikov-2019-Optimal, Bubeck-2019-Near, Jiang-2020-Optimal}. However, a bisection search procedure is necessary in each iteration of all these methods~\citep{Monteiro-2013-Accelerated, Gasnikov-2019-Optimal, Bubeck-2019-Near, Jiang-2020-Optimal}, and the total number of subproblems solved at each bisection step is bounded by a logarithmic factor in the given precision. On the other hand, the missing factor in the complexity estimate for the accelerated cubic regularized Newton's method is in the order of $O(\epsilon^{-1/21})$. As demonstrated by~\citet{Nesterov-2019-Implementable}, the additional logarithmic factors in the complexity bound of A-HPE method will definitely overshadow its tiny superiority in the convergence rate. From the practical efficiency point of view, the acceleration second-order scheme presented in \citet{Nesterov-2008-Accelerating} and~\citet{Monteiro-2013-Accelerated} are not easily implementable, since they assume the knowledge of some Lipschitz constant of the Hessian. To alleviate this,~\citet{Jiang-2020-Unified} incorporated an adaptive strategy \citep{Cartis-2011-Adaptive-I, Cartis-2011-Adaptive-II} into Nesterov's approach \citep{Nesterov-2008-Accelerating, Nesterov-2019-Implementable}, and further relaxed the criterion for solving each sub-problem while maintaining the same iteration complexity for convex optimization. However, the deterministic second-order method, e.g., the one proposed by~\citet{Jiang-2020-Unified}, may be computationally costly as it requires the full second-order information.

The seminal work of~\citet{Robbins-1951-Stochastic} triggered a burst of research interest on developing stochastic first-order methods. Regarding the second-order methods (in particular Newton's method), there has been a recent intensive research attention in designing their stochastic variants suitable for large-scale applications, e.g.\ stochastic quasi-Newton methods~\citep{Byrd-2016-Stochastic, Schraudolph-2007-Stochastic}, stochastic cubic regularization method~\citep{Tripuraneni-2018-Stochastic}, randomized cubic regularization method~\citep{Doikov-2018-Randomized}, stochastic trust region method~\citep{Blanchet-2019-Convergence}, stochastic line search method~\citep{Paquette-2020-Stochastic}, Hessian sketching~\citep{Pilanci-2017-Newton, Cormode-2019-Iterative} and sub-sampling methods~\citep{Agarwal-2017-Second, Byrd-2011-Use, Bollapragada-2019-Exact, Erdogdu-2015-Convergence, Kylasa-2018-GPU, Liu-2017-Inexact, Yao-2018-Inexact, Li-2019-SubNewton, Roosta-2019-SubNewton, Xu-2016-SubNewton}. {Note that all the works for finding the global minimizers on \textit{sub-sampling methods} assume that all the component functions are convex. }In terms of cubic regularized Newton's method for non-convex optimization, {the adaptive regularization algorithms with inexact evaluation for both  function and derivatives are considered in \cite{Bellavia-2019-Adaptive}.}~\citet{Kohler-2017-SubSample} proposed a uniform sub-sampling strategy to approximate the Hessian matrix and the gradient, however, in each step of the algorithm the sample size for the approximation is unknown until the cubic subproblem in this iteration is solved.~\citet{Xu-2019-Newton} resolved this issue by conducting appropriate uniform and non-uniform sub-sampling strategies to construct Hessian approximations within the cubic regularization scheme and~\citet{Yao-2018-Inexact} further proposed inexact variants of trust region and adaptive cubic regularization methods, which can be implemented in practice without any knowledge of unknowable problem-related quantities. {The adaptive cubic regularization methods with dynamic
	inexact Hessian information for finite-sum minimization and stochastic optimization
	are studied in \cite{Bellavia-2021-Finite-Sum} and \cite{Bellavia-2020-Complexity} respectively.}~\citet{Zhang-2018-SubNewton} managed to incoporate sub-sampling strategies into the variance reduction techniques. Under the framework of more general probabilistic models, some probabilistic convergence results for cubic regularization methods were established in~\citet{Cartis-2018-Probabilistic}. For convex optimization,~\citet{Ghadimi-2017-Second} proposed an accelerated Newton's method with cubic regularization using inexact second-order information and such information could be obtained from a subsample strategy. However, their algorithm fails to retain the
the iteration bound of $O(\epsilon^{-1/3})$, although the acceleration is indeed observed in the numerical experiments. Another recent work by \citet{Ye-2020-Nesterov} resorted to {Nesterov's} acceleration to improve the convergence performance of second-order methods (approximate Newton), including regularized sub-sampled Newton, and provided nice empirical evaluation results. However, the acceleration is only achieved when the objective function is strongly convex. After the first version of this paper was published online,~\citet{Song-2019-Inexact} in the meanwhile studied an accelerated inexact proximal cubic regularized Newton's method that allows a composite objective: the sum of a smooth and a nonsmooth convex function. Their algorithm still assumes the knowledge of the Lipschitz constant, and has the iteration bound of $O(\epsilon^{-1/3})$ in the sense of expectation. It is worth noting that both~\citet{Ghadimi-2017-Second} and~\citet{Song-2019-Inexact} assume the approximated Hessian is pre-given and satisfy certain nice properties that need be used in the analysis. In that regard, our algorithm allows a dynamic adjustment of the sample size of the approximated Hessian, which leads to a low per-iteration computational cost at certain stage of the algorithm. The resulting computational benefits are evidently observed (and some of which will be reported in this work) in the process of our numerical experiments.

\subsection{Notations and Organization}
Throughout the paper, we denote vectors by bold lower case letters, e.g., $\x$, and matrices by regular upper case letters, e.g., $X$. The transpose of a real vector $\x$ is denoted as $\x^\top$. For a vector $\x$, and a matrix $X$, $\left\|\x\right\|$ and $\left\|X\right\|$ denote the $\ell_2$ norm and the matrix spectral norm, respectively. $\nabla f(\x)$ and $\nabla^2 f(\x)$ are respectively the gradient and the Hessian of $f$ at $\x$, and $\BI$ denotes the identity matrix. For two symmetric matrices $A$ and $B$, $A \succeq B$ indicates that $A-B$ is symmetric positive semi-definite. The subscript, e.g., $\x_i$, denotes iteration counter. $\log(\alpha)$ denotes the natural logarithm of a positive number $\alpha$. $\frac{0}{0}=0$ is imposed for non-uniform setting. The inexact Hessian is denoted by $H(\x)$, but for notational simplicity, we also use $H_i$ to denote the inexact Hessian evaluated at the iterate $\x_i$ in iteration $i$, i.e., $H_i\triangleq H(\x_i)$. The calligraphic letter $\SCal$ denotes a collection of indices from $\left\{1,2,\ldots, n\right\}$, with potential repeated items and its cardinality is denoted by $\left|\SCal\right|$.

The rest of the paper is organized as follows. In Section \ref{Sec:Pre}, we introduce the assumptions underlying this paper, and the tradeoff between the sample size and the accuracy of the resulting approximated Hessian. Then the sub-sampling accelerated cubic regularized Newton's method is presented in Section \ref{Section:SAARC}. The probabilistic and worst case iteration complexity of the algorithm are analyzed in Section \ref{Section:Probabilistic} and Section \ref{Section:worst-case} respectively. In Section \ref{Section5:Experiment}, we present some preliminary numerical results on solving regularized logistic regression, where the effect of acceleration together with low per-iteration computational cost are clearly observed. The details of  most proofs can be found in the appendix.

\section{Preliminaries} \label{Sec:Pre}
In this section, we first introduce the main definitions and assumptions used in the paper, and then present two lemmas on the construction of the inexact Hessian in random sampling.
\subsection{Assumptions} \label{Section2:Assumption}
Throughout this paper, we refer to the following definition of $\epsilon$-optimality.
\begin{definition}($\epsilon$-optimality) 
Given $\epsilon\in\left(0,1\right)$, $\x\in\br^d$ is said to be an $\epsilon$-optimal solution to problem~\eqref{Prob:main}, if
\begin{equation}\label{result:optimality}
f(\x) - f^\star \leq \epsilon,	\quad \text{or} \quad \| \nabla f(\x) \|^2 \leq \epsilon.
\end{equation}
\end{definition}

To proceed, 
we make the following standard assumption regarding the gradient and Hessian of the objective function $f$.
\begin{assumption}\label{Assumption-Objective-Gradient-Hessian}
The objective function $f(\x)$ in problem~\eqref{Prob:main} is convex and twice differentiable. Each of $f_j(\x)$ is possibly nonconvex but twice differentiable with the gradient and the Hessian being both Lipschitz continuous, i.e., there are $0<L_j, \rho_j<\infty$ such that for any $\x, \y\in\br^d$ we have
\begin{align}
\|\nabla f_j(\x) - \nabla f_j(\y)\| \leq &  \ L_j \|\x-\y\|, \label{Def:Lipschitz-Gradient} \\
\|\nabla^2 f_j(\x) - \nabla^2 f_j(\y)\| \leq &  \ \rho_j \|\x-\y\|. \label{Def:Lipschitz-Hessian}
\end{align}
	
\end{assumption}
In the rest of the paper, we define $L=\max_j L_j>0$ and $\bar{L} = \frac{1}{n}\sum_{j=1}^n L_j>0$, and $\bar{\rho} = \frac{1}{n}\sum_{j=1}^n \rho_j$.
A consequence of \eqref{Def:Lipschitz-Gradient} is that
\begin{equation}\label{Bounded-Hessian}
\|\nabla^2 f_j(\x)\| \leq L_j \; \mbox{ and } \; \|\nabla^2 f(\x)\| \leq \bar L \; \quad \forall\; \x \in \br^d.
\end{equation}
We consider the following approximation of $f$ evaluated at $\x_i$ with cubic regularization \citep{Cartis-2011-Adaptive-I, Cartis-2011-Adaptive-II} in our algorithm:
\begin{equation}\label{prob:SARC}
m(\s;\x_i,\sigma_i) = f(\x_i) + \s^\top\nabla f(\x_i) + \frac{1}{2}\s^\top H(\x_i)\s + \frac{1}{3}\sigma_i\left\|\s\right\|^3,
\end{equation}
where $\sigma_i>0$ is a regularized parameter adjusted in the process as the algorithm progresses.
{ Let {$\x_0$} be the starting point of our algorithm and {$\x^\star$} be an optimal solution of problem~\eqref{Prob:main}. Then sub-level set $\mathcal L(\x_0, \sigma_0) := \{ \x_0 + \s \in \br^d \mid m(\s; \x_0,\sigma_0 ) \le m(0; \x_0,\sigma_0 ) = f(\x_0)  \}$ at $\x_0$ with regularization parameter $\sigma = \sigma_0$ is bounded as the function $m(\s; \x_0,\sigma_0 )$ is coercive. Hence, there is some $D \geq 1$ such that 
\begin{equation}\label{bounded-level-set}
\max_{\x \in \mathcal L(\x_0, \sigma_0)} \|\x - \x^\star\| \le D.
\end{equation}}

\subsection{Random Sampling} \label{Section2:Sampling}
When each $f_i$ in~\eqref{Prob:main} is convex, random sampling has been proven to be {a} very effective approach in reducing the computational cost; see~\citet{Erdogdu-2015-Convergence, Roosta-2019-SubNewton, Bollapragada-2019-Exact, Xu-2016-SubNewton}. In this subsection, we show that such random sampling can indeed be employed for the setting considered in this paper.

Suppose that the probability distribution of the sampling over the index set $\{1,2,\ldots,n\}$ is $\p=\{p_i\}_{i=1}^{i=n}$ with $\Prob(\xi=i)=p_i\geq 0$ for $i=1,2,\ldots,n$. Let $\SCal$ and $|\SCal|$ denote the sample collection and its cardinality respectively, and define
\begin{equation}\label{Condition: SSHessian}
{\tilde{H}(\x)} = \frac{1}{n|\SCal|}\sum_{j\in\SCal} \frac{1}{p_j}\nabla^2 f_j(\x),
\end{equation}
to be the sub-sampled Hessian. When $n$ is very large, such random sampling can significantly reduce the per-iteration computational cost as $|\SCal| \ll n$. There are two sampling strategies in the literature: uniform sampling and non-uniform sampling  \cite{Xu-2019-Newton}. In the following, we review some technical result of each approach demonstrating how many samples are required to get an approximated Hessian within a given accuracy. The first one is to sample $\{1,2,\ldots,n\}$ uniformly, i.e., $p_i=1/n$. The lemma below is a simple restatement of~\citet[Lemma~16]{Xu-2019-Newton}.
\begin{lemma}\label{Lemma:Uniform-Sample}
Suppose Assumption~\ref{Assumption-Objective-Gradient-Hessian} holds for problem~\eqref{Prob:main}. A uniform sampling with or without replacement is performed to form the sub-sampled Hessian. That is for $\x\in\br^d$, the matrix ${\tilde{H}(\x)}$ is constructed from \eqref{Condition: SSHessian} with $p_j=\frac{1}{n}$ and sample size
\begin{equation*}
|\SCal| \geq \Theta^U({\hat \epsilon},\delta) := \frac{16L^2}{{\hat \epsilon}^2} \cdot \log\left(\frac{2d}{\delta}\right)
\end{equation*}
for given $0<{\hat \epsilon}, \delta<1$, where $L$ is defined as in Assumption~\ref{Assumption-Objective-Gradient-Hessian}. Then we have
\begin{equation*}
\Prob(\|{\tilde{H}(\x)}- \nabla^2 f(\x)\| \geq {\hat \epsilon}) < \delta.
\end{equation*}
\end{lemma}
In case problem \eqref{Prob:main} is endowed with more structures, then some more ``informative" distribution may be constructed as opposed to simple uniform sampling.
For instance, if it is in the form of \eqref{Prob:special}, then we can introduce a bias in the probability distribution and pick those \textit{relevant} $f_i$'s carefully. As suggested in \cite{Xu-2019-Newton}, we construct
\begin{equation}\label{nonuniform-probability}
p_j = \frac{| {f_j^{\prime \prime}}(\sa_j^\top \x)|\|\sa_j\|^2}{\sum_{k=1}^n |{f_j^{\prime \prime}}(\sa_k^\top \x)|\|\sa_k\|^2}, 
\end{equation}
where the absolute values are taken since ${  f_j}$ is possibly nonconvex. Next we restate~\citet[Lemma 17]{Xu-2019-Newton} below about the sampling complexity for the construction of approximated Hessian of problem~\eqref{Prob:special}.
\begin{lemma}\label{Lemma:NonUniform-Sample}
Suppose Assumption~\ref{Assumption-Objective-Gradient-Hessian} holds for problem~\eqref{Prob:special}. A non-uniform sampling is performed to form the sub-sampled Hessian. That is for $\x\in\br^d$, the matrix ${\tilde{H}(\x)}$ is constructed from \eqref{Condition: SSHessian} with $\p$ as defined in~\eqref{nonuniform-probability} and sample size
\begin{equation*}
|\SCal| \geq \Theta^N({\hat \epsilon},\delta) := \frac{4\bar{L}^2}{{\hat \epsilon}^2} \cdot \log\left(\frac{2d}{\delta}\right),
\end{equation*}
for given $0<{\hat \epsilon}, \delta<1$, where $\bar{L}$ is defined in Assumption~\ref{Assumption-Objective-Gradient-Hessian}. Then, we have
\begin{equation*}
\Prob(\|{\tilde{H}(\x)}- \nabla^2 f(\x)\| \geq {\hat \epsilon}) < \delta.
\end{equation*}
\end{lemma}
Compared to Lemma~\ref{Lemma:Uniform-Sample}, {computing the sampling probability in Lemma~\ref{Lemma:NonUniform-Sample} requires going through all data points, whose computational effort amounts to evaluating the full gradient once. However, the sampling complexity mainly comes from the sample size rather than the sampling probability. This is because the computational cost of forming the approximated Hessian matrix heavily depends on the sample size and such matrix is frequently sampled in our algorithm (i.e., sampled in every step of our algorithm). Moreover, the sample size provided by Lemma~\ref{Lemma:NonUniform-Sample} could be smaller as $\bar{L}\leq L$. In this case, the non-uniform sampling is preferable where the distributions of $L_j$ are skewed, i.e., some $L_j$ are much larger than the others and $\bar{L} \ll L$.  This advantage has been demonstrated by the practical performance of randomized coordinate descent method and sub-sampled Newton method~\citep{Qu-2016-Coordinate-I, Qu-2016-Coordinate-II, Xu-2016-SubNewton}. Therefore, in this case, the computational savings stems from the smaller sample size dominates the cost of computing the sampling probability for the non-uniform sampling scheme.}

Note that in the above two lemmas, the sample size is only proportional to the log of the failure probability, and thus we can use a very small failure per-iteration probability to guarantee the solution quality without increasing the sample size significantly. {Although, the sample sizes in Lemma \ref{Lemma:Uniform-Sample} and \ref{Lemma:NonUniform-Sample} is dependent on the Lipschitz constant, its exact value is not necessarily required and any of its upper bound would work. In addition, we provide worst-case analysis in Section \ref{Section:worst-case}, which guarantees the convergence of our algorithm regardless of the estimation quality of the Lipschitz constant.}

\section{Accelerated Adaptive Cubic Regularization of Newton's Method with Uniform and Nonuniform Sub-Sampling}
 \label{Section:SAARC}
\subsection{The Algorithm}
Now we propose the accelerated sub-sampling adaptive cubic regularization method as presented in Algorithm \ref{Algorithm:SAARC}. In particular, we adopt a two-phase scheme, where the acceleration is implemented in Phase II. It is worth noting that a direct extension of the accelerated cubic regularization method under inexact Hessian information fails to maintain the theoretical convergence property \citep{Ghadimi-2017-Second}. Therefore, the two-phase scheme is necessary to establish the accelerated rate of convergence, where the first phase serves the purpose of finding a good starting point for acceleration. Phase I and Phase II are referred to as simple sub-sampling adaptive subroutine (SSAS) and accelerated sub-sampling adaptive subroutine (ASAS), respectively, and the details are described in Algorithm \ref{Alg:SSAS} and Algorithm \ref{Alg:ASAS}. 
{In particular, note that there are two counters of iterations in Algorithm \ref{Alg:ASAS}. One is $j$ that counts the generic iterations, and the other one is $l$ for the successful iterations. In each generic iteration $j$ of Algorithm \ref{Alg:ASAS}, 
an approximate minimizer of the cubic model is computed. If the generic iteration is successful and early stopping is not activated, the auxiliar model is adaptively minimized in an inner loop for acceleration, and then the current approximation of the cubic model and the counter $l$ for the successful iterations are updated. Otherwise, the current approximation is left unchanged and the coefficient $\sigma$ of the cubic regularization term is reduced.
}In the following, we elaborate on some key steps of these algorithms.

\begin{algorithm}[!t]
\begin{algorithmic}
\STATE \textbf{Input:} $\x_0\in\br^d$, $\sigma_0 \geq \sigma_{\min} > 0$, $\tau_0 > 0$, $\gamma_2 > \gamma_1 > 1$, $\gamma_3 > 1$, $\eta > 0$, $\delta_0\in(0,1)$, $\kappa_\theta\in(0,1)$, initial tolerance of Hessian approximation $\epsilon_0 = \min\{1, \frac{\left\|\nabla f(\x_0)\right\|}{3}\}$, and {tolerance of the approximate solution $\epsilon$.}
\STATE \textbf{Phase I (SSAS):} $[\x_0^{I}, \sigma_0^{I}, \epsilon^{I}, T_1] = \textsf{SSAS}(\x_0, \sigma_0, \epsilon_0, {\epsilon}, \gamma_1, \gamma_2, \delta_0, \kappa_{\theta})$.
\IF{$\| \nabla f(\x_0^{I}) \|^2 \le \epsilon$}
\STATE terminate Algorithm 1 [{\it early stop}], and return $\x_{out} = \x_0^{ASAS}$.
\ENDIF
\STATE \textbf{Phase II (ASAS):} $[\x_{out}, T_2, T_3] = \textsf{ASAS}(\x_0^{I}, \sigma_0^{I}, \sigma_{\min}, \epsilon^{I}, {\epsilon}, \varsigma_0, \gamma_1, \gamma_2, \gamma_3, \eta, \delta_0, \kappa_{\theta})$.
\STATE {Let $T = T_1 + T_2 + T_3$ [{\it record the total iteration number}].}
\STATE \textbf{Output:} an $\varepsilon$-optimal solution $\x_{out}$ and $T$.			
\end{algorithmic}\caption{Accelerated Subsampling Adaptive Cubic Regularized Newton's Method}\label{Algorithm:SAARC}
\end{algorithm}

{
\textbf{Constructing the cubic model}: 
Given the iteration point $\x_i$, cubic regularized parameter $\sigma_i$, tolerance of Hessian approximation $\epsilon_i$, the accuracy of the optimal solution $\epsilon$, and overall failure probability $\delta_0$. We adopt the notation $\textsf{Cubic}(\x_i, \sigma_i, \epsilon_i, \epsilon, \delta_0)$ to denote the generator of the cubic model as follows:
\begin{equation}\label{cubic-generator}
\textsf{Cubic}(\x_i, \sigma_i, \epsilon_i, \epsilon, \delta_0) \rightarrow f(\x_i) + \s^\top\nabla f(\x_i) + \frac{1}{2}\s^\top H(\x_i)\s + \frac{1}{3}\sigma_i\left\|\s\right\|^3,
\end{equation}
where ${H}(\x_i) = \tilde{H}(\x_i) + \epsilon_i {\BI}$, and $\tilde{H}(\x_i)$ is constructed according to \eqref{Condition: SSHessian} with sample size $|\SCal| \geq \Theta^U(\epsilon_i, \delta_0 \epsilon^{1/3})$ for uniform sampling ($\Theta^N(\epsilon_i, \delta_0 \epsilon^{1/3})$ for non-uniform sampling) {such that
$$\|\nabla^2 f(\x_i) - \tilde{H}(\x_i) \| \leq  \epsilon_i $$
with probability at least $1 - \delta_0 \epsilon^{1/3}$. If the above inequality holds, the approximated Hessian $H(\x_i)$ in the cubic model is also a good estimation, i.e.,
\begin{equation}\label{Hession-Approximation}
\|\nabla^2 f(\x_i) - H(\x_i)\| \leq {2}\epsilon_i.
\end{equation}
In addition, the convexity of $f$ implies that
\begin{equation}\label{H-PSD}
H(\x_i) = \tilde{H}(\x_i) + \epsilon_i \BI \succeq \nabla^2 f(\x_i) - \epsilon_i \BI + \epsilon_i \BI = \nabla^2 f(\x_i) \succeq 0.
\end{equation}
with probability at least  $1 - \delta_0 \epsilon^{1/3}$.
} 
}

\textbf{Solving the cubic model}:  Recall that $m(\s;\x_i,\sigma_i)$ is the cubic $\sigma_i$-regularized function at $\x_i$ defined in \eqref{prob:SARC}.
In each iteration, we approximately solve
\begin{equation} \label{prob:ARC}
\s_i\approx\argmin_{\s\in\br^d} \ m(\s;\x_i,\sigma_i),
\end{equation}
where $m(\s;\x_i,\sigma_i)$ is defined in \eqref{prob:SARC} and the symbol ``$\approx$'' is quantified as follows:
\begin{condition}\label{Cond:Approx_Subprob_SAARC}
	We call $\s_i$ to be an approximate solution of the subproblem -- denoted as $\s_i \approx \argmin_{\s\in\br^d} \ m(\s;\x_i,\sigma_i)$ -- for $\min_{\s\in\br^d} \ m(\s;\x_i,\sigma_i)$, if $m(\s_i;\x_i,\sigma_i) \le m(0;\x_i,\sigma_i) = f(\x_i)$ and
	\begin{equation}\label{Eqn:Approx_Subprob_SAARC}
	\|\nabla f(\x_i) + H(\x_i)\s_i + \sigma_i\|\s_i\|\s_i\| \leq \kappa_\theta \min \{\|\s_i\|^2, \| \nabla f(\x_i)\|\},
	\end{equation}
	where $0< \kappa_\theta < 1 $ is a pre-specified  constant.
\end{condition}
\begin{algorithm}[!t]
	\begin{algorithmic}
		\STATE \textbf{Initialization:} Let the total iteration count $i=0$.		
		\STATE {Generate cubic model $m(\s;\x_0,\sigma_0)$ with $\textsf{Cubic}(\x_0, \sigma_0, \epsilon_0, \epsilon, \delta_0)$ according to  \eqref{cubic-generator}}. 
		\STATE Let $\theta_0 = -1$.
		\WHILE{$\theta_i \leq 0$ }
		\STATE Compute $\s_i\in\br^d$ such that $\s_i \approx \argmin_{\s\in\br^d} \ m(\s;\x_0,\sigma_0)$ according to Condition \ref{Cond:Approx_Subprob_SAARC};
		\STATE Compute $\theta_i=m(\s_i; \x_i, \sigma_i)-f(\x_i+\s_i)$.
		\IF{$\theta_i>0$ [{\it successful iteration}]}
		\STATE Let $\x_{i+1}=\x_i+\s_i$, $\sigma_{i+1}= \sigma_i$,  
		\begin{equation*}
		\epsilon_{i+1} = \min\left\{\frac{\|\nabla f(\x_{i+1})\|}{6}, \ {  \epsilon_{0}}\right\} \mbox{[{\it  update tolerance of Hessian approximation}]}.
		\end{equation*}
		Update $i=i+1$.
		\ELSE
		\STATE $\x_{i+1}=\x_i$, $\sigma_{i+1} \in [\gamma_1\sigma_i, \gamma_2\sigma_i]$, $\epsilon_{i+1}=\epsilon_i$, update $i = i+1$.
		\ENDIF
		\ENDWHILE
		\STATE {Let $T_1 = i$ [{\it  record the iteration number}]}.
		\STATE Return $\x_i, \sigma_i, \epsilon_i$.
	\end{algorithmic}\caption{$\textsf{SSAS}(\x_0, \sigma_0, {\epsilon_0, \epsilon}, \gamma_1, \gamma_2, \delta_0, \kappa_\theta)$} \label{Alg:SSAS}
\end{algorithm}

\textbf{Solving the auxiliary model}: {The acceleration in Phase II is achieved by minimizing an auxiliary model:
\begin{equation*}
\psi_l(\z) = \psi_{l-1}(\z) + \frac{l(l+1)}{2}\left(f(\bar{\x}_{l-1})+\left(\z - \bar{\x}_{l-1}\right)^\top\nabla f(\bar{\x}_{l-1})\right) + \frac{1}{6}(\varsigma_{l}-\varsigma_{l-1}) \|\z - \bar{\x}_0\|^3,
\end{equation*}	
with $\psi_0(\z)={f}(\x_0) + \frac{1}{6}\varsigma_0\|\z-\bar{\x}_0\|^3$. To be specific, $\psi_l(\z)$ is used as a bridge to establish the iteration bounds in Theorem \ref{Theorem:SAARC} and Theorem \ref{Theorem:W-AARC}. Moreover, the minimizer of auxiliary model $\psi_l(\z)$ has a closed-form expression (see \citet{Nesterov-2008-Accelerating} and~\cite{Jiang-2020-Unified}):  $\bar{\x}_0-\sqrt{\frac{2}{\varsigma_l \|\nabla\ell_l(\z)\|}} {\nabla\ell_l(\z)}$ with 
$$\ell_l(\z) = \ell_{l-1}(\z) + \frac{l(l+1)}{2}\left(f(\bar{\x}_{l-1})+\left(\z - \bar{\x}_{l-1}\right)^\top\nabla f(\bar{\x}_{l-1})\right)\quad\mbox{and}\quad\ell_0(\z)= {f}(\x_0).$$
}

\begin{algorithm}[!t]
\begin{algorithmic}
\STATE \textbf{Initialization:} Let the total iteration count $i=0$, the successful iteration count $l=0$, the iteration count $k=0$ of updating $\varsigma_l$, and $\bar{\x}_0 = \x_{0}$.
\STATE Construct $\psi_0(\z) = f(\bar{\x}_0) + \frac{1}{6}\varsigma_0\|\z -\bar{\x}_0 \|^3$, and compute $\z_0= \argmin_{\z\in \br^d} \psi_0(\z)$. 
\STATE Let $\y_0 = \frac{1}{4}\bar{\x}_0 + \frac{3}{4}\z_0$ [{\it  generate base point for the cubic model}].
\STATE { Generate cubic model $m(\s;\y_0,\sigma_0)$ with $\textsf{Cubic}(\y_0, \sigma_0, \epsilon_0, \epsilon, \delta_0)$ according to \eqref{cubic-generator}}.
\FOR{$j=0,1,2,\ldots,$ {until convergence}}
\STATE Compute $\s_{j}\approx\argmin_{\s\in\br^d} m(\s;\y_l,\sigma_{j})$ using Condition \ref{Cond:Approx_Subprob_SAARC}, and $\rho_{j} = -\frac{\s_{j}^\top \nabla f(\y_l+\s_{j})}{\|\s_{j}\|^3}$.
\IF{$\rho_{j}\geq\eta$ [{\it successful iteration}]}
\STATE $\bar \x_{l+1} = \x_{j+1}=\y_l+\s_{j}$, $\sigma_{j+1} \in [\sigma_{\min}, \sigma_{j}]$, and let
\begin{equation*}
\epsilon_{j+1} = \min\left\{\frac{\|\nabla f({\y_{l}})\|}{4}, \ {  \epsilon_{0}}\right\}. \mbox{ [{\it  update tolerance of Hessian approximation}]}
\end{equation*}
\IF{$\|\nabla f(\x_{j+1}) \|^2 \le \epsilon$}
\STATE terminate Algorithm \ref{Alg:ASAS} [{\it  early stop}], and return $\x_{out} = \x_{j+1}$.
\ENDIF
\STATE Set $l=l+1$, $\varsigma_l = \varsigma_{l-1}$, and compute $\z_l = \argmin_{\z\in\br^d} \ \psi_l(\z)$.
\WHILE{$\psi_l(\z_l) < \frac{l(l+1)(l+2)}{6} f(\bar{\x}_l) $}
\STATE Set $\varsigma_l=\gamma_3\varsigma_l$, and $k = k+1$ [{\it  record the count of updating $\varsigma_l$}]. 
\STATE Update $\psi_l(\z) = \psi_{l-1}(\z) + \frac{l(l+1)}{2}[f(\bar \x_{l})+(\z-\bar \x_{l})^\top \nabla f(\bar \x_{l})] + \frac{1}{6}(\varsigma_l-\varsigma_{l-1})\|\z-\bar{\x}_0\|^3$. 
\STATE Compute $\z_l=\argmin_{\z\in\br^d}\psi_l(\z)$.
\ENDWHILE
\STATE Compute $\y_l=\frac{l}{l+3}\bar{\x}_l + \frac{3}{l+3}\z_l$ [{\it  generate base point for the cubic model}].
\STATE {  Generate cubic model $m(\s;\y_l,\sigma_{j+1})$ with $\textsf{Cubic}(\y_l, \sigma_{j+1}, \epsilon_{j+1}, \epsilon, \delta_0)$ according to \eqref{cubic-generator}}.
\ELSE
\STATE Let $\x_{j+1}=\x_{j}$, $\sigma_{j+1} \in [\gamma_1\sigma_j, \gamma_2\sigma_j]$;
\ENDIF
\ENDFOR
\STATE { Let $T_2 = j+1$ [{\it  record the number of solving the cubic subproblem}].
\STATE 	Let $T_3 = k$ [{\it  record the total number of updating $\varsigma_l$}].}
\STATE Return $\x_{out} = \bar \x_l$, $T_2$ and $T_3$.
\end{algorithmic} \caption{$\textsf{ASAS}(\x_0, \sigma_0, \sigma_{\min}, \epsilon_0, {\epsilon}, \varsigma_0, \gamma_1, \gamma_2, \gamma_3, \eta, \delta, \kappa_{\theta})$}\label{Alg:ASAS}
\end{algorithm}

\subsection{Overview of the Analysis}
{ Recall in our algorithms that $T_1$ is the total number of iterations in Phase I, $T_2$ is the total number of solving the cubic model in Phase II, and $T_3$ is the total count of updating the parameter $\varsigma_l$ in the auxiliary model. Then the iteration complexity is established if we are able to bound $T_1$, $T_2$ and $T_3$.
Before presenting the technical analysis, we sketch some major steps as follows, 
\begin{enumerate}
\item Upper bound $T_1$ in Lemma \ref{Lemma:SAARC-T1} (Lemma \ref{Lemma:W-AARC-T1} for worst case analysis) .
\item Prove $T_2$ to be $|\SucCal|$ multiplied by some factors in Lemma \ref{Lemma:SAARC-T2} (Lemma \ref{Lemma:W-AARC-T2} for worst case analysis), where $\SucCal = \{j\leq T_2: j \; \text{is a successful iteration}\}$ is the index set of all successful iterations in Phase II.
\item Upper bound $T_3$ in Lemma \ref{Lemma:SAARC-T3} (Lemma \ref{Lemma:W-AARC-T3} for worst case analysis).
\item Upper bound $|\SucCal|$ in Theorem \ref{Theorem:SAARC} (Theorem \ref{Theorem:W-AARC} for worst case analysis).
\item Put the pieces together, and prove the iteration bound in Theorem \ref{Theorem:SAARC-Iteration-Complexity} (Theorem \ref{Thm:W-AARC-Main} for worst case analysis).
\end{enumerate} 
For probabilistic iteration complexity, the quantities in the bounds of $T_1$, $T_2$ and $T_3$ (except for $|\SucCal |$) depend only on the problem parameters (see Lemmas \ref{Lemma:SAARC-T1}--\ref{Lemma:SAARC-T3}) thus do not affect the magnitude of the iteration bound. While for the worst-case iteration complexity, those quantities in Lemmas \ref{Lemma:W-AARC-T1}--\ref{Lemma:W-AARC-T3} depend on $\epsilon$, i.e., the solution accuracy. Such dependence will eventually deteriorate the iteration complexity bound, 
such as the one presented in Theorem \ref{Thm:W-AARC-Main}.}

\section{Probabilistic Iteration Complexity} \label{Section:Probabilistic}
Now we are in a position to provide iteration complexity analysis for Algorithm~\ref{Algorithm:SAARC}. We shall show that Algorithm \ref{Algorithm:SAARC} retains the iteration complexity of $O(\epsilon^{-1/3})$, the same as that of the non-adaptive version \citep{Nesterov-2008-Accelerating}, even though the sub-problem is now only solved approximately with sub-sampled Hessian. {In the following, to highlight the flow of our analysis, we shall present the contents of the key technical lemmas while relegating the proofs to the appendix.}

We first provide Lemma \ref{Lemma:SAARC-T1} and Lemma \ref{Lemma:SAARC-T2}, which {describe} the relation between the total iteration {number} in Algorithm \ref{Algorithm:SAARC} and the amount of successful iterations $|\SucCal|$ in Phase II.
\begin{lemma}\label{Lemma:SAARC-T1}
Suppose in each iteration $i$ of Algorithm \ref{Alg:SSAS},  the sub-sampled Hessian $\tilde{H}(\x_i)$ satisfies 
	\begin{equation}\label{Sub-Sampled-Hession-Approximation}
	\|\nabla^2 f(\x_i) - \tilde H(\x_i)\| \leq \epsilon_i.
	\end{equation} 
Denoting $\bar{\sigma}^P_1={ \max\{\sigma_0, \ 3\gamma_2 + 0.5\bar\rho \gamma_2, \ \gamma_2  (L+\epsilon_0+ \kappa_\theta +   \bar \rho   )      \} }$, it holds that
\begin{equation*}
T_1 \leq \left\lceil 1+ \frac{{1}}{\log\gamma_1} \log\left(\frac{\bar{\sigma}^P_1}{\sigma_{\min}}\right)\right\rceil.
\end{equation*}
\end{lemma}	

\begin{lemma}\label{Lemma:SAARC-T2}
Suppose in each iteration $j$ of Algorithm \ref{Alg:ASAS},
 the sub-sampled Hessian $\tilde{H}(\x_j)$ satisfies $\|\nabla^2 f(\x_j) - \tilde H(\x_j)\| \leq \epsilon_j$. Denoting 
\begin{equation*}
\bar{\sigma}^P_2 = { \max\left\{\bar{\sigma}^P_1, \ \frac{\gamma_2\bar{\rho}}{2} + \gamma_2\kappa_\theta+\gamma_2\eta+2\gamma_2, \ \gamma_2 L + \gamma_2\epsilon_0+ \gamma_2\bar\rho + 3\gamma_2 \kappa_\theta + 2\gamma_2\eta \right\}} > 0, 
\end{equation*}
and $\SucCal$ to be the set of successful iterations in Algorithm \ref{Alg:ASAS}, it holds that
\begin{equation*}
T_2 \leq \left\lceil 1+ \frac{2}{\log\gamma_1} \log\left(\frac{\bar{\sigma}^P_2}{\sigma_{\min}}\right)\right\rceil |\SucCal|.
\end{equation*}
\end{lemma}
Then we estimate an upper bound on $T_3$: the total counts updating $\varsigma>0$.
\begin{lemma}\label{Lemma:SAARC-T3}
Suppose in each iteration $j$ of Algorithm \ref{Alg:ASAS},
 the sub-sampled Hessian $\tilde{H}(\x_j)$ satisfies $\|\nabla^2 f(\x_j) - \tilde H(\x_j)\| \leq \epsilon_j$. It holds that
\begin{equation}\label{Inequality:induction}
\psi_l(\z_l)\ge\frac{l(l+1)(l+2)}{6} f(\bar{\x}_l)
\end{equation}
when $\varsigma_l \ge \bar \varsigma^P := 8\eta^{-2}(\bar{\rho} + (2\kappa_\theta + 2)L + 2\bar{\sigma}^P_2 +\kappa_\theta+1)^3$, which further implies
\begin{equation*}
T_3 \leq \left\lceil 1 + \frac{1}{\log\left(\gamma_3\right)}\log \left[\frac{8\left({ \frac{\bar{\rho}}{2}+ 2\kappa_\theta + L  + 2\bar{\sigma}^P_2+1}\right)^3}{\eta^2 \, \varsigma_0}\right]\right\rceil.
\end{equation*}
\end{lemma}
In the rest of this section, the total number of iterations of the two subroutines (i.e.\ Algorithm \ref{Alg:SSAS} and Algorithm \ref{Alg:ASAS}) is referred to as the iteration complexity of Algorithm \ref{Algorithm:SAARC}. {To continue our analysis, we prove the following theorem to provide a bound on
the number of successful iterations in Algorithm \ref{Alg:ASAS}.}
\begin{theorem}\label{Theorem:SAARC}
Suppose in each iteration $i$ of Algorithm \ref{Algorithm:SAARC},
 the sub-sampled Hessian $\tilde{H}(\x_i)$ satisfies \eqref{Sub-Sampled-Hession-Approximation}.  Then the sequence $\{\bar{\x}_l, \ l=0,1,\ldots\}$ generated by Algorithm \ref{Alg:ASAS} satisfies
\begin{eqnarray*}
& & \frac{l(l+1)(l+2)}{6}f(\bar{\x}_l) \leq \psi_l(\z_l) \leq \psi_l(\z) \\
& \leq & \frac{l(l+1)(l+2)}{6}f(\z) + {8\kappa_\theta D^3} +  \frac{\bar L +\epsilon_0}{2}\|\z-\x_{0}\|^2 + \frac{\bar \sigma^P_{1}}{3}\|\z-\x_{0}\|^3 + \frac{\varsigma_l}{6}\|\z-\bar{\x}_0\|^3.
\end{eqnarray*}
\end{theorem}
\begin{proof}
The proof is based on mathematical induction. The base case $l =0$ corresponds to $f(\bar{\x}_0) = \psi_0(\z_0)$, which follows from the definition of $\psi_0(\z)$. It suffices to show the inequality on the right hand side. Denote $ \x_0 \in \br^d$ as the initial iterate in Algorithm \ref{Alg:SSAS}, $\bar{\x}_0$ is the output returned by Algorithm {\ref{Alg:SSAS}} and $\bar \s^m_{0}$ as a global minimizer of $m(\s, \x_{0}, \sigma_0^{ASAS})$ over $\br^d$. We also note that for each $\sigma_i$ in Algorithm \ref{Alg:SSAS}, $\sigma_i \ge \sigma_{\min}$ and thus $\mathcal{L}(\x_0,\sigma_i) \subseteq \mathcal{L}(\x_0, \sigma_{\min})$. Then, noting $\bar \x_0 = \x_0 + \bar \s_0$, by \eqref{bounded-level-set} one has
\begin{equation}\label{Bounded-levelset-s}
\|\x_0 + \bar \s_0 - \x^\star\| \leq D \quad \mbox{and} \quad \|\x_0 + \bar \s_0^m - \x^\star\| \leq D.
\end{equation}
Furthermore, by the criterion of successful iteration in Algorithm \ref{Alg:SSAS},
\begin{equation*}
f(\bar{\x}_0) \leq m(\bar \s_{0}, \x_{0}, \sigma_0^{ASAS}) = (m(\bar\s_{0}, \x_{0}, \sigma_0^{ASAS}) - m(\bar\s^m_{0}, \x_{0}, \sigma_0^{ASAS})) + m(\bar\s^m_{0}, \x_{0}, \sigma_0^{ASAS}).
\end{equation*}
{Since $\|\nabla^2 f(\x_i) - \tilde{H}(\x_i)\| \leq {\epsilon_i}$ for all $i$, and $f$ is convex, we have \eqref{H-PSD} holds and $H(\x_i) \succeq 0$.}
{Besides, we note that $\nabla^2\left(\|\s\|^3\right)= 3\left(\|\s\| \cdot \BI + \s \s^{\top}\right)\succeq 0$.} Therefore, $m(\s, \x_{0},\sigma_0^{ASAS})$ is convex and we have
\begin{eqnarray*}
& & m(\bar \s_{0}, \x_{0}, \sigma_0^{ASAS}) - m(\bar \s^m_{0}, \x_{0}, \sigma_0^{ASAS}) \\
& \leq & (\nabla f(\x_0) + H(\x_0){\bar\s_0} + \sigma_0^{ASAS}{\|\bar\s_0\|\cdot\bar \s_0})^\top(\bar\s_{0} - \bar\s^m_{0}) \\
& \leq & \|\nabla f(\x_0) + H(\x_0){\bar\s_0} +  \sigma_0^{ASAS}{\|\bar\s_0\|\cdot\bar \s_0}\| \cdot \|\bar\s_{0} - \bar\s^m_{0}\| \\
& {\eqref{Eqn:Approx_Subprob_SAARC} \above 0pt \leq} & \kappa_\theta\|\bar\s_{0}\|^2\| \bar\s_{0} - \bar\s^m_{0}\| \\
& \leq & \kappa_\theta\|\bar\s_{0} + \x_0 - \x^\star - (\x_0 - \x^\star)\|^2\|\bar\s_{0} + \x_0 - \x^\star - (\bar\s^m_{0} + \x_0 - \x^\star)\| \\
& {\eqref{Bounded-levelset-s}\, \eqref{bounded-level-set} \above 0pt \leq } & {8\kappa_\theta D^3}.
\end{eqnarray*}
On the other hand, we also have
\begin{eqnarray*}
& & m(\bar \s^m_{0}, \x_{0}, \sigma_0^{ASAS}) \\
& = & f(\x_{0}) + (\bar \s^m_{0})^\top \nabla f(\x_{0}) + \frac{1}{2} (\bar \s^m_{0})^\top H(\x_{0}) \bar \s^m_{0} + \frac{1}{3}\sigma_0^{ASAS}\| \bar \s^m_{0}\|^3  \\
& \leq & f(\x_{0}) + (\z-\x_{0})^\top \nabla f(\x_{0}) + \frac{1}{2} (\z-\x_{0})^\top \nabla^2 f(\x_{0})(\z-\x_{0}) + \frac{\epsilon_{0}}{2}\|\z-\x_{0}\|^2  + \frac{\sigma_0^{ASAS}}{3}\|\z-\x_{0}\|^3 \\
& \leq & f(\z) + \frac{\bar{L}}{2}\|\z-\x_{0}\|^2 + \frac{\epsilon_0}{2}\|\z-\x_{0}\|^2 + \frac{\sigma_0^{ASAS}}{3}\|\z-\x_{0}\|^3 \\
& \leq & f(\z) + \frac{\bar L +\epsilon_0}{2}\|\z-\x_{0}\|^2 + \frac{\bar \sigma_{1}^P}{3}\|\z-\x_{0}\|^3 ,
\end{eqnarray*}
where the second inequality is due to the convexity of $f$ and \eqref{Bounded-Hessian}. Therefore,
\begin{equation*}
\psi_0(\z) = f(\bar{\x}_0)  + \frac{1}{6}\varsigma_0\|\z-\bar{\x}_0\|^3 \leq f(\z) + {8\kappa_\theta D^3} + \frac{\bar L +\epsilon_0}{2}\|\z-\x_{0}\|^2 + \frac{\bar \sigma_{1}^P}{3}\|\z-\x_{0}\|^3 + \frac{1}{6}\varsigma_0\|\z-\bar{\x}_0\|^3.
\end{equation*}
Now suppose that the theorem is proven for some $l\geq 1$. Let us consider the case of $l+1$:
\begin{eqnarray*}
& & \psi_{l+1}(\z_{l+1}) \ \leq \ \psi_{l+1}(\z) \\
& = & \psi_{l}(\z) + \frac{(l+1)(l+2)}{2}[f(\bar{\x}_l)+(\z-\bar{\x}_{l})^\top \nabla f(\bar{\x}_l)]+ \frac{1}{6}(\varsigma_{l+1}-\varsigma_l)\|\z-\bar{\x}_0t\|^3 \\
& \leq & \frac{l(l+1)(l+2)}{6}f(\z) + {8\kappa_\theta D^3} + \frac{\bar L +\epsilon_0}{2}\|\z-\x_{0}\|^2 + \frac{\bar \sigma_{1}^P}{3}\|\z-\x_{0}\|^3  + \frac{\varsigma_l}{6}\|\z-\bar{\x}_0\|^3 \\
& & + \frac{(l+1)(l+2)}{2}[f(\bar{\x}_l)+(\z-\bar{\x}_l)^\top\nabla f(\bar{\x}_l)]+ \frac{1}{6}(\varsigma_{l+1}-\varsigma_l)\|\z-\bar{\x}_0\|^3 \\
& \leq & \frac{(l+1)(l+2)(l+3)}{6}f(\z) + {8\kappa_\theta D^3} + \frac{\bar L +\epsilon_0}{2}\|\z-\x_{0}\|^2 + \frac{\bar \sigma_{1}^P}{3}\|\z-\x_{0}\|^3 + \frac{\varsigma_{l+1}}{6}\|\z-\bar{\x}_0\|^3,
\end{eqnarray*}
where the last inequality is due to convexity of $f(\z)$. On the other hand, noting the way that $\psi_{l+1}(\z)$ is updated we have $\frac{(l+1)(l+2)(l+3)}{6}{f}(\bar{\x}_{l+1}) \leq \psi_{l+1}(\z_{l+1})$. The theorem is thus proven by induction.
\end{proof}
After establishing Theorem~\ref{Theorem:SAARC}, the iteration complexity of Algorithm \ref{Algorithm:SAARC} readily follows.
\begin{theorem}\label{Theorem:SAARC-Iteration-Complexity}
Let $\epsilon$ be the accuracy of optimality, {$\epsilon_i$ be the tolerance of sub-sampled Hessian approximation in \eqref{Sub-Sampled-Hession-Approximation}} for iteration $i$, and {$\delta_0$} be the probability that inequality \eqref{Sub-Sampled-Hession-Approximation} fails for at least one iteration. When Algorithm~\ref{Algorithm:SAARC} runs
\begin{eqnarray*}
T & = & \left\lceil 1 + \frac{2}{\log(\gamma_1)}\log \left(\frac{\bar{\sigma}^P_1}{\sigma_{\min}}\right) \right\rceil + \left\lceil 1 + \frac{2}{\log(\gamma_1)}\log\left(\frac{\bar{\sigma}^P_2}{\sigma_{\min}}\right)\right\rceil \left\lceil\left(\frac{C^P}{\epsilon}\right)^{\frac{1}{3}}\right\rceil \\
& & + \left\lceil \frac{1}{\log\left(\gamma_3\right)}\log \left[\frac{{8\eta^{-2}(0.5\bar{\rho} + 2\kappa_\theta + L +{ \epsilon_0} + 2\bar{\sigma}^P_2+2)^3}}{\eta^2 \, \varsigma_0}  \right] +1\right\rceil \\
& = & \OCal(\epsilon^{-1/3})
\end{eqnarray*}
iterations (including the successful iterations to update $\varsigma$), then with probability $1-{\delta_0}$ we have $f(\x_{out}) - f^\star \leq \epsilon$, where $C^P = D^3({48\kappa_\theta} + 2{\bar\sigma}^P_1 + \gamma_3\bar\varsigma^P) + 3D^2(\bar L +\epsilon_0)$.
\end{theorem}
\begin{proof}
{
Under the probability assumption of Theorem~\ref{Theorem:SAARC} and by taking $\z=\x^\star$, we have that
\begin{eqnarray*}
	& & \frac{l(l+1)(l+2)}{6}f(\bar{\x}_l) \\
	& \leq & \frac{l(l+1)(l+2)}{6}f(\x^\star) + {8\kappa_\theta D^3} +  \frac{\bar L +\epsilon_0}{2}\|\x^\star-\x_{0}\|^2 + \frac{{\bar \sigma}^P_1}{3}\|\x^\star-\x_{0}\|^3 + \frac{\varsigma_{l}}{6}\|\x^\star-\bar{\x}_0\|^3\\
	& \leq & \frac{l(l+1)(l+2)}{6}f^\star + \left({8\kappa_\theta} + \frac{{\bar\sigma}^P_1}{3} + \frac{\gamma_3\bar\varsigma^P}{6}\right)D^3 + \left(\frac{\bar L +\epsilon_0}{2}\right)D^2.
\end{eqnarray*}
Rearranging the terms yields that
$$
f(\bar \x_l) - f^* \le \frac{C^P}{l(l+1)(l+2)} {<} \frac{C^P}{l^3}.
$$
Recall that $l$ is the count of successful iterations and $\SucCal$ is  the index set of all successful iterations in Algorithm \ref{Alg:ASAS}.
{Then $|\SucCal| = l < \left(\frac{C^P}{\epsilon}\right)^{1/3}$ whenever $f(\bar{x}_{l}) - f^* \ge \epsilon$. Therefore, by choosing $T_2 = \left\lceil 1 + \frac{2}{\log(\gamma_1)}\log\left(\frac{\bar{\sigma}^P_2}{\sigma_{\min}}\right)\right\rceil\left\lceil\left(\frac{C^P}{\epsilon}\right)^{1/3}\right\rceil$
and by Lemma \ref{Lemma:SAARC-T2}, we have $l = |\SucCal| \ge \left\lceil\left(\frac{C^P}{\epsilon}\right)^{1/3}\right\rceil $, which further implies $f(\bar{x}_{l}) - f^* < \epsilon$.}
Denote $\hat T = T_1 + T_2$ to be the total number of iterations that generates sub-sampled Hessian in Algorithm \ref{Algorithm:SAARC}.
{Then combining the choice of $T_2$ and Lemma \ref{Lemma:SAARC-T1} yields that}
\begin{equation*}
\hat T = \left\lceil 1 + \frac{2}{\log(\gamma_1)}\log \left(\frac{\bar{\sigma}^P_1}{\sigma_{\min}}\right) \right\rceil+ \left\lceil 1 + \frac{2}{\log(\gamma_1)}\log\left(\frac{\bar{\sigma}^P_2}{\sigma_{\min}}\right)\right\rceil\left\lceil\left(\frac{C^P}{\epsilon}\right)^{1/3}\right\rceil = \OCal(\epsilon^{-1/3})
\end{equation*}	
To ensure an overall accumulative success probability of $1-{\delta_0}$ for the entire $\hat T$ iterations, the per-iteration failure probability is set as $1-\sqrt[\hat T]{1-\delta_0} = \OCal(\delta_0/\hat T)= \OCal(\delta_0\epsilon^{1/3})$; see~\citet{Xu-2019-Newton} for more details. Therefore, by setting $\hat \epsilon= \epsilon_i$ and $\delta = \delta_0 \epsilon^{1/3}$ in Lemma \ref{Lemma:Uniform-Sample} (or Lemma \ref{Lemma:NonUniform-Sample}) and 
we have that $\|\nabla^2 f(\x_i) - {\tilde H(\x_i)}\| \leq \epsilon_i $ for all $i \le T_1 + T_2$ with probability $1 - \delta_0$. As a result, the probability assumption in Theorem~\ref{Theorem:SAARC} is satisfied, and the conclusion follows {from the choice of $\hat T$ and Lemma \ref{Lemma:SAARC-T3}}.

}	
\end{proof}

\section{Worst-Case Iteration Complexity}\label{Section:worst-case}
In this section, we {consider the case where the accuracy requirement of the sub-sampled Hessian is not satisfied, and} assume that each component function $f_i$ in $f$ of problem \eqref{Prob:main} is convex. For any $H(\x)$ constructed in Algorithm \ref{Alg:SSAS} and Algorithm \ref{Alg:ASAS}, {noting that $\epsilon_0$ is the upper bound of the tolerance of all Hessian approximations in the algorithms}, it holds that
\begin{equation}\label{Property: matr-H}
H(\x) \succeq 0 \quad \mbox{and}\quad {\| H(\x)\| \le \frac{1}{n|\SCal|}\sum_{j\in\SCal} \frac{1}{p_j} \| \nabla^2 f_j(\x_i) \| + \epsilon_0 \overset{\eqref{Bounded-Hessian}}\le  L + \epsilon_0 }. 
\end{equation}
{In the following, to highlight the flow of our analysis, we shall present the lemmas key to our analysis  but relegate their proofs to the appendix.}
\begin{lemma}\label{Lemma:W-AARC-T1} Suppose $\|\nabla f(\x_i)\|^2 > {\epsilon}$ in each iteration $i$ of Algorithm \ref{Alg:SSAS}. Denoting
\begin{equation*}
\bar{\sigma}^W_1=\max\left\{\sigma_0, \ \frac{3\gamma_2 L(4L + \epsilon_0)}{(1-\kappa_\theta)\sqrt{\epsilon}} \right\} > 0,
\end{equation*}
we have
\begin{equation*}
T_1 \le \left\lceil 1+\frac{2}{\log\left(\gamma_1\right)}\log\left(\frac{\bar{\sigma}^W_1}{\sigma_{\min}}\right) \right\rceil.
\end{equation*}
\end{lemma}
\begin{lemma}\label{Lemma:W-AARC-T2}
Suppose $ \| \nabla f(\x_j) \|^2 > {\epsilon}$ in each iteration $j$ of Algorithm \ref{Alg:ASAS}. Denoting
{\small \begin{equation}\label{eqn:bar-sigma2-W}
\bar{\sigma}_2^W = \max \left\{\bar{\sigma}_1^W, \gamma_2 \frac{(3L+2\epsilon_0)(2L+\epsilon_0)+2\sqrt{\epsilon}(1-\kappa_\theta)(\kappa_\theta+\eta)+(2L+\epsilon_0)\sqrt{(3L+2\epsilon_0)^2 +\sqrt{\epsilon}(1-\kappa_\theta)(\kappa_\theta+\eta)}}{2 \sqrt{\epsilon}(1-\kappa_\theta)} \right\}.
\end{equation}}
we have
\begin{equation*}
T_2 \leq \left\lceil 1+ \frac{2}{\log\left(\gamma_1\right)} \log \left(\frac{\bar{\sigma}_2^ {W}}{\sigma_{\min}}\right)\right\rceil |\SucCal|.
\end{equation*}
\end{lemma}
Now we are ready to estimate an upper bound of $T_3$: the total counts of successfully updating $\varsigma>0$.
\begin{lemma}\label{Lemma:W-AARC-T3} 
Suppose in each iteration $j$ of Algorithm \ref{Alg:ASAS}, we have $\|\nabla f(\x_j)\|^2 > {\epsilon}$ for all $0 \le j \le T_2$. Then inequality \eqref{Inequality:induction} holds if
\begin{equation}\label{bound-varsigma-W}
\varsigma_l \geq \bar \varsigma^W :=  \frac{8}{\eta^2}\left((2L+\epsilon_0)\cdot\frac{(L+\epsilon_0) +\sqrt{(L+\epsilon_0)^2+4{\bar\sigma}^W_{2}\sqrt{\epsilon}(1-\kappa_\theta)}}{2\sqrt{\epsilon}(1-\kappa_\theta)}+\bar{\sigma}^W_2 + \kappa_\theta\right)^3,
\end{equation}
{where $\bar{\sigma}^W_2$ is defined in \eqref{eqn:bar-sigma2-W}, and it} further implies that
\begin{equation*}
T_3 \leq \left\lceil \frac{1}{\log\left(\gamma_3\right)}\log \left[\frac{8}{\eta^2\varsigma_0}\left((2L+\epsilon_0) \cdot \frac{(L+\epsilon_0) +\sqrt{(L+\epsilon_0)^2+4{\bar \sigma}^W_{2}\sqrt{\epsilon}(1-\kappa_\theta)}}{2\sqrt{\epsilon}(1-\kappa_\theta)} + \bar{\sigma}^W_2 + \kappa_\theta \right)^3\right] + 1 \right\rceil
\end{equation*}
\end{lemma}
In the rest of this section, we refer the combined number of iterations of the two subroutines (Algorithm \ref{Alg:SSAS} and Algorithm \ref{Alg:ASAS}) as the iteration count for Algorithm \ref{Algorithm:SAARC}.
\begin{theorem}\label{Theorem:W-AARC}
Suppose that every component function $f_i$ in $f$ of problem \eqref{Prob:main} is convex and in each iteration $i$ of Algorithm \ref{Algorithm:SAARC}, we have $\|\nabla f(\x_j)\|^2 > {\epsilon}$ for all $j$. Then the sequence $\{\bar{\x}_l, \ l=0,1,\ldots\}$ generated by Algorithm \ref{Alg:ASAS} satisfies
\begin{eqnarray*}
& & \frac{l(l+1)(l+2)}{6}f(\bar{\x}_l) \leq \psi_l(\z_l) \leq \psi_l(\z) \\
& \leq & \frac{l(l+1)(l+2)}{6}f(\z) + {8 \kappa_\theta D^3} + \frac{L +\epsilon_0}{2}\|\z-\x_{0}\|^2 + \frac{\bar\sigma^W_{1}}{3}\|\z-\x_{0}\|^3 + \frac{\varsigma_l}{6}\|\z-\bar{\x}_0\|^3.
\end{eqnarray*}
\end{theorem}
\begin{proof} The proof is almost identical to that of Theorem \ref{Theorem:SAARC} (which is based on mathematical induction) except the following estimation on $m(\bar \s^m_{0}, \x_{0}, \sigma_0^{ASAS})$, where $\bar \s^m_{0}$ is a global minimizer of $m(\s, \x_{0}, \sigma_0^{ASAS})$ over $\br^d$:
\begin{eqnarray*}
m(\bar \s^m_{0}, \x_{0}, \sigma_0^{ASAS})
& = & f(\x_{0}) + (\bar \s^m_{0})^\top \nabla f(\x_{0}) + \frac{1}{2} (\bar \s^m_{0})^\top H(\x_{0}) \bar \s^m_{0} + \frac{1}{3}\sigma_0^{ASAS}\|\bar \s^m_{0}\|^3 \\
& { \eqref{Property: matr-H}\above 0pt  \leq } & f(\x_{0}) + (\z-\x_{0})^\top \nabla f(\x_{0}) + \frac{1}{2} (L+\epsilon_0) \left\|\z-\x_{0}\right\|^2  + \frac{\sigma_0^{ASAS}}{3}\|\z-\x_{0}\|^3  \\
& \leq & f(\z) + \frac{ L +\epsilon_0}{2}\|\z-\x_{0}\|^2 + \frac{{\bar \sigma}^W_1}{3}\|\z-\x_{0}\|^3 ,
\end{eqnarray*}	
where the second inequality is due to the convexity of $f$. Then, by replacing the estimation of $m(\bar \s^m_{0}, \x_{0}, \sigma_0^{ASAS})$ with the inequality above, the conclusion readily follows.
\end{proof}
After establishing Theorem~\ref{Theorem:W-AARC} and denoting
{\footnotesize \begin{equation}\label{eqn:bar-sigma-W}
\bar{\sigma}^W := \max\left\{\sigma_0, \frac{3\gamma_2 L(4L + \epsilon_0)}{(1-\kappa_\theta)}, \gamma_2\frac{(3L+2\epsilon_0)(2L+\epsilon_0)+2(1-\kappa_\theta)(\kappa_\theta+\eta)+(2L+\epsilon_0)\sqrt{(3L+2\epsilon_0)^2+(1-\kappa_\theta)(\kappa_\theta+\eta)}}{2 (1-\kappa_\theta)}\right\},
\end{equation}}
the iteration complexity of Algorithm \ref{Algorithm:SAARC} readily follows.
\begin{theorem}\label{Thm:W-AARC-Main}
Suppose every component function $f_i$ in $f$ of problem \eqref{Prob:main} {is convex}, and let $0<\epsilon<1$ sufficiently small. The Algorithm~\ref{Algorithm:SAARC} returns a solution $\x_{out}$ such that either $ \| \nabla f(\x_{out}) \|^2 \le {\epsilon}$ or $f(\x_{out}) -f^* \le \epsilon$, at an iteration no more than 
\begin{eqnarray*}
T & \leq & \left\lceil 1 + \frac{2}{\log(\gamma^W_1)}\log \left(\frac{\bar{\sigma}^W}{\sigma_{\min}} \epsilon^{-\frac{1}{2}} \right) \right\rceil + \left\lceil 1 + \frac{2}{\log(\gamma_1)}\log\left(\frac{\bar{\sigma}^W}{\sigma_{\min}}\epsilon^{-\frac{1}{2}}\right)\right\rceil \left\lceil (C^W)^{\frac{1}{3}}\cdot\epsilon^{-\frac{5}{6}} \right\rceil \\
& & + \left\lceil \frac{1}{\log(\gamma_3)}\log\left[\left((2L+\epsilon_0) \cdot \frac{(L+\epsilon_0) +\sqrt{(L+\epsilon_0)^2+4{\bar \sigma}^W(1-\kappa_\theta)}}{2(1-\kappa_\theta)}  + \bar{\sigma}^W + \kappa_\theta \right)^{3} \frac{8 D^3}{\eta^2\varsigma_0}\epsilon^{-\frac{3}{2}} \right]+1 \right\rceil\\
& = & \OCal(\epsilon^{-5/6}\log(\epsilon^{-1})),
\end{eqnarray*}
where
\begin{equation}\label{constant:W-C}\small
C^W = {12\kappa_\theta D^3} + 3(L +\epsilon_0)D^2 + 2\bar\sigma^W D^3 + \frac{8D^3}{\eta^2} \left((2L+\epsilon_0) \cdot \frac{(L+\epsilon_0) +\sqrt{(L+\epsilon_0)^2+4{\bar \sigma}^W(1-\kappa_\theta)}}{2(1-\kappa_\theta)}  + \bar{\sigma}^W + \kappa_\theta \right)^{3}
\end{equation}	
{and $\bar{\sigma}^W$ is defined in \eqref{eqn:bar-sigma-W}.}	
\end{theorem}
\begin{proof} Suppose Algorithm~\ref{Algorithm:SAARC}  does not stop early, i.e.,  we have $\|\nabla f(\x_j)\|^2 > \epsilon$ in every iteration $j$. Recall that $l= 0,1,\ldots$ is the count of successful iterations in Algorithm \ref{Alg:ASAS}. Applying the inequality in Theorem~\ref{Theorem:W-AARC} with $\z=\x^\star$ we have
\begin{equation*}
\frac{l(l+1)(l+2)}{6}(f(\bar{\x}_l) - f(\x^\star)) \leq {8\kappa_\theta D^3} + \left(\frac{L +\epsilon_0}{2}\right)D^2 + \frac{\bar\sigma^W_{1}}{3}D^3 + \frac{\varsigma_l}{6}D^3
\end{equation*}
Note that $\varsigma_l \leq \gamma_3 \bar \varsigma^W$ and { $\bar \varsigma^W$ has the magnitude of $\epsilon^{-\frac{3}{2}}$ in \eqref{bound-varsigma-W}. In addition,
$\bar \sigma^W_{1}$ that is defined in Lemma \ref{Lemma:W-AARC-T1}, is also dependent on $\epsilon$.} { The above inequality implies that
\begin{equation}{\label{W-l3-epsilon}}
f(\bar{\x}_l)-f(\x^\star) \leq \frac{C^W}{l(l+1)(l+2)}\cdot\epsilon^{-\frac{3}{2}} < \frac{C^W}{l^3}\cdot\epsilon^{-\frac{3}{2}} 
\end{equation}
with $C^W$ defined in \eqref{constant:W-C}. 
 Recall that $\SucCal$ is  the index set of all successful iterations in Algorithm \ref{Alg:ASAS}.
Then $|\SucCal| = l < \left(\frac{C^W}{\epsilon^{5/2}}\right)^{1/3}$ whenever $f(\bar{x}_{l}) - f^* \ge \epsilon$. Therefore, by choosing $T_2 = \left\lceil 1 + \frac{2}{\log(\gamma_1)}\log\left(\frac{\bar{\sigma}^W_2}{\sigma_{\min}}\right)\right\rceil\left\lceil\left(\frac{C^W}{\epsilon^{5/2}}\right)^{1/3}\right\rceil$
and Lemma \ref{Lemma:W-AARC-T2}, we must have $l = |\SucCal| \ge \left\lceil\left(\frac{C^W}{\epsilon^{5/2}}\right)^{1/3}\right\rceil $, which further implies $f(\bar{x}_{l}) - f^* < \epsilon$. Finally, the upper bound on $T$ follows by combining this result with lemmas \ref{Lemma:W-AARC-T1} and \ref{Lemma:W-AARC-T3}.}
\end{proof}

To conclude this section, we remark that if we adopt a stronger early stop criterion of $\|\nabla f(\x_{j+1}) \| \le \epsilon$ in Algorithm~\ref{Algorithm:SAARC}, then the iteration bound in Theorem \ref{Thm:W-AARC-Main} will change to $\OCal(\epsilon^{-4/3}\log(\epsilon^{-1/2}))$. This is because in this case, the factor $\sqrt{\epsilon}$ in the bound of $\bar \varsigma^W$ from \eqref{bound-varsigma-W} is replaced by $\epsilon$. As a result, the quantity $\frac{C^W}{l^3}\cdot\epsilon^{-\frac{3}{2}} $ in \eqref{W-l3-epsilon} is adapted to $\frac{C^W}{l^3}\cdot\epsilon^{-{3}} $.
{Then we can let $T_2 = \left\lceil 1 + \frac{2}{\log(\gamma_1)}\log\left(\frac{\bar{\sigma}^W_2}{\sigma_{\min}}\right)\right\rceil\left\lceil\left(\frac{C^W}{\epsilon^{4}}\right)^{1/3}\right\rceil$ and guarantee $l \ge \lceil (C^W)^{\frac{1}{3}} \epsilon^{-\frac{4}{3}}\rceil$ by Lemma \ref{Lemma:W-AARC-T1}, which further implies $f(\bar{x}_{l}) - f^* < \epsilon$.  The iteration bound of  $\OCal(\epsilon^{-4/3}\log(\epsilon^{-1/2}))$ follows from this result, lemmas \ref{Lemma:W-AARC-T1} and \ref{Lemma:W-AARC-T3}.}

\section{Numerical Experiments}\label{Section5:Experiment}
We shall demonstrate the efficacy of the proposed method by presenting some computational results on different genres of real data. Experimental results on regularized logistic regression confirm that our algorithm is suitable for solving large-scale statistical learning problems and at least competitive with other algorithms. In addition, all { eight} data sets are selected from the LIBSVM collection\footnote[1]{The LIBSVM collection is available at https://www.csie.ntu.edu.tw/$\sim$cjlin/libsvmtools/datasets} in which their statistics are summarized in Table~\ref{Tab:Dataset-Stats}, and all algorithms are implemented using Python 3.5 on a MacBook Pro running with Mac OS High Sierra 10.13.6 and 16GB memory.
\begin{table}[t]
\centering
\caption{The Statistics of Eight LIBSVM Datasets}\label{Tab:Dataset-Stats}
\begin{tabular}{|r|r|r|c|} \hline
Name & Instances No. & Features No. & Processing  \\ \hline
\textsf{SUSY} & 5,000,000 & 18 & Done by \cite{Baldi-2014-Searching} \\
\textsf{covtype} & 581,012 & 54 & Transformed from multiclass by~\cite{Collobert-2002-Parallel}. \\
\textsf{phishing} & 11,055 & 68 & Binary encoding and length-normalized \\
\textsf{w8a} & 49,749 & 300 & Rescaled to a unit vector \\
\textsf{gisette} & 7,000 & 5,000 & Feature-wisely rescaled within $[-1, 1]$ \\
\textsf{rcv1} & 20,242 & 47,236 & Only training data used \\
\textsf{real-sim} & 72,309 & 20,958 & Vikas Sindhwani for the SVMlin project \\ \hline
\end{tabular}
\end{table}
\begin{figure*}[!ht]
\centering
\includegraphics[width=0.40\textwidth]{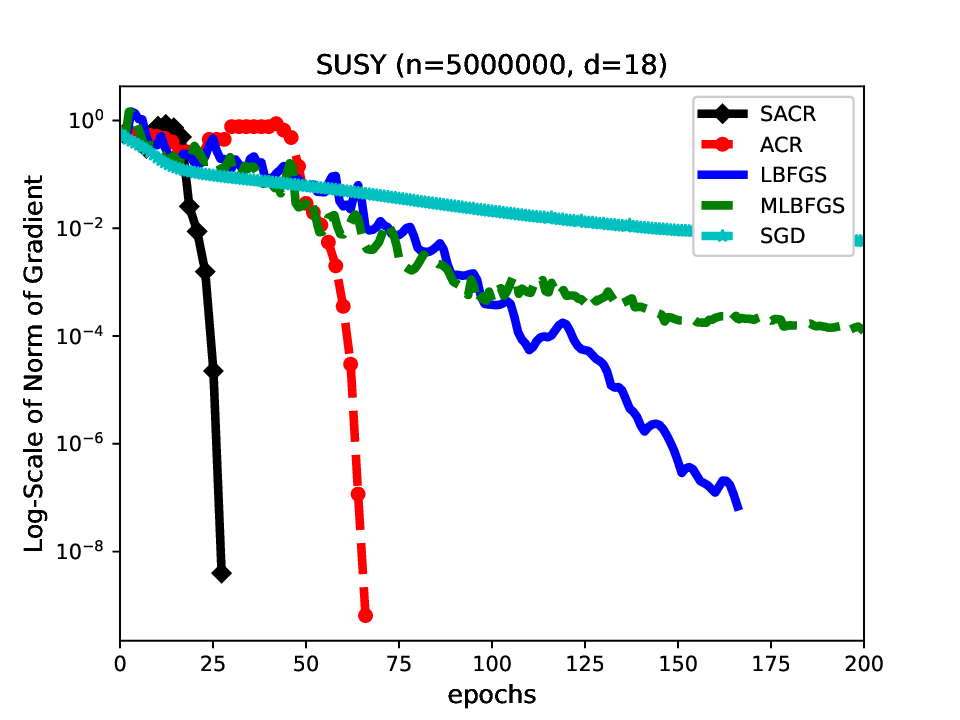}
\includegraphics[width=0.40\textwidth]{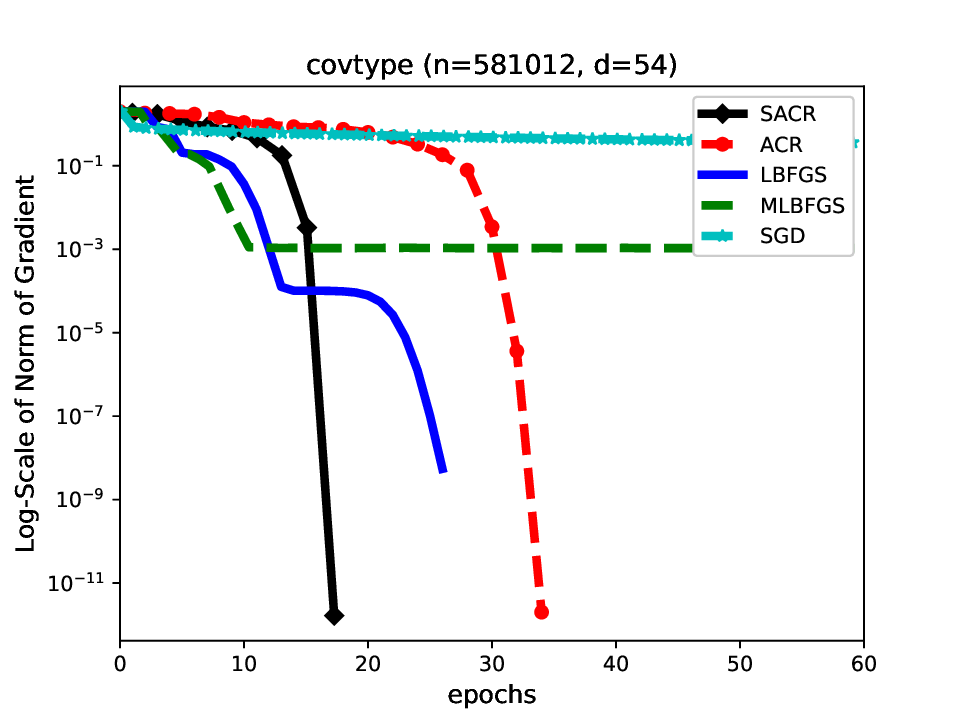}
\includegraphics[width=0.40\textwidth]{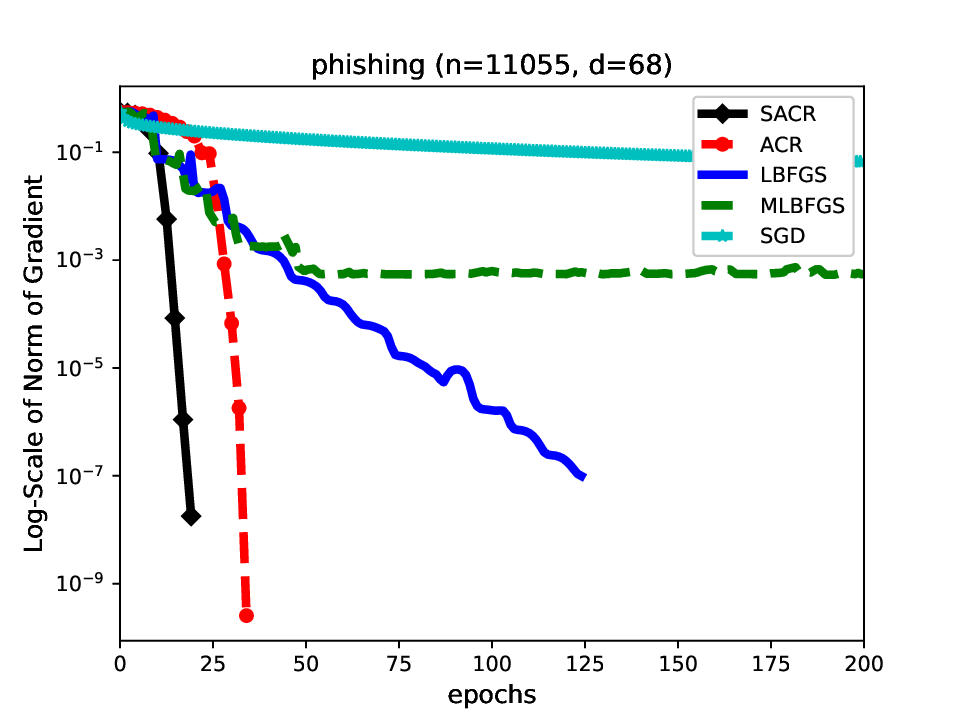}
\includegraphics[width=0.40\textwidth]{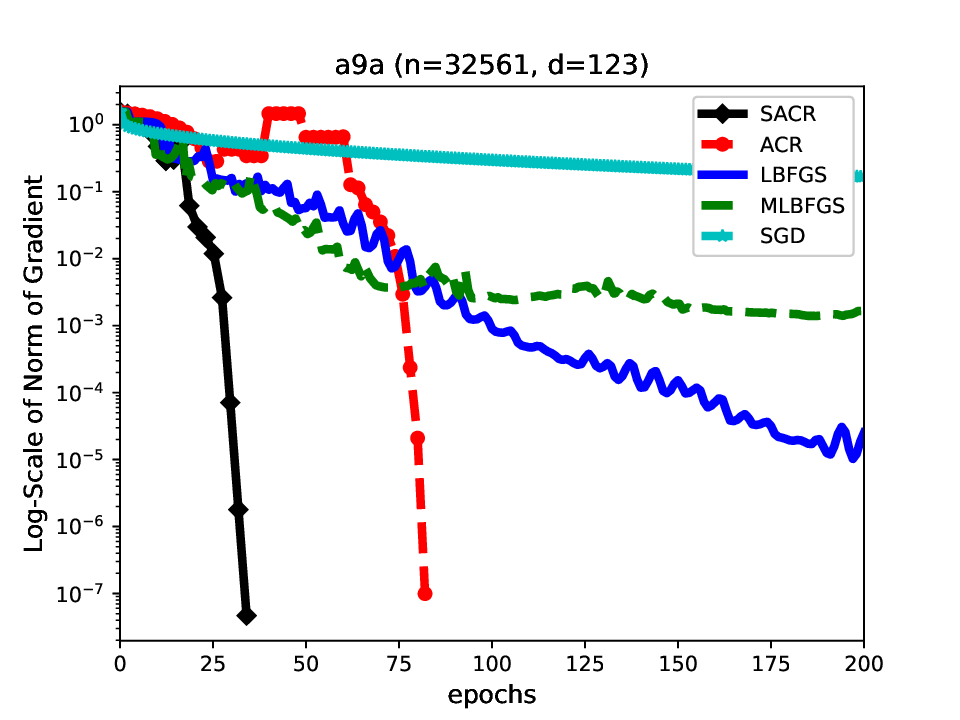}
\includegraphics[width=0.40\textwidth]{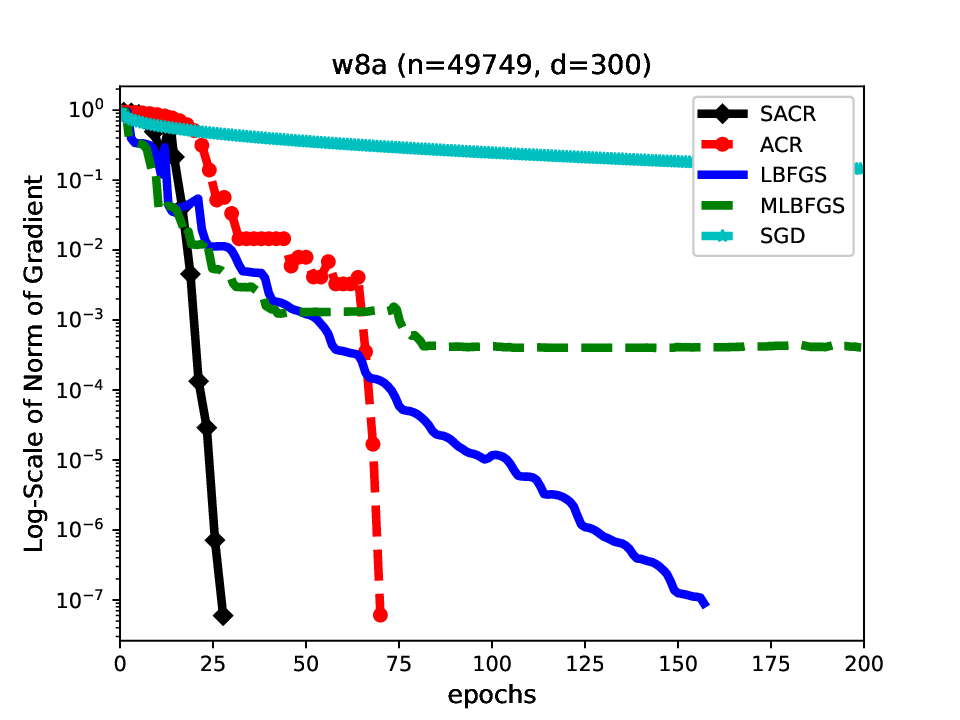}
\includegraphics[width=0.40\textwidth]{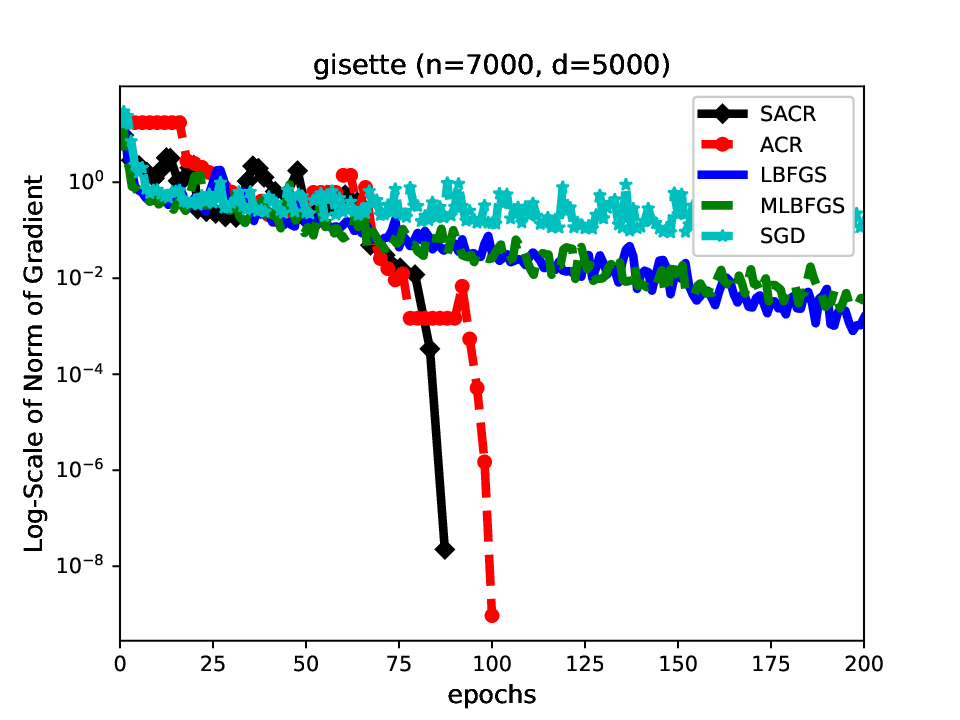}
\includegraphics[width=0.40\textwidth]{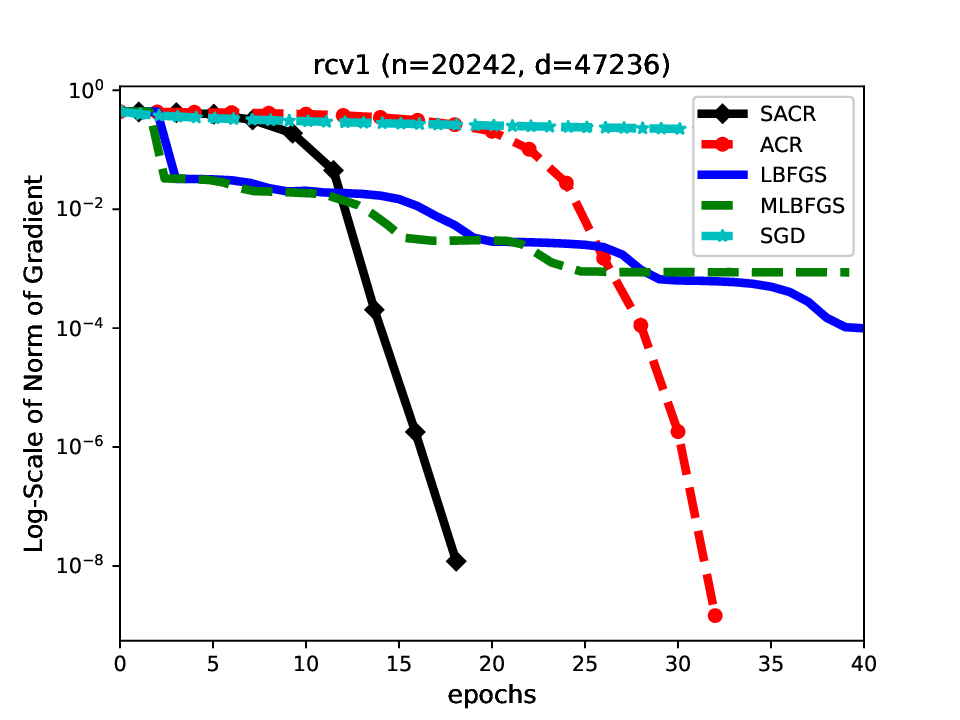}
\includegraphics[width=0.40\textwidth]{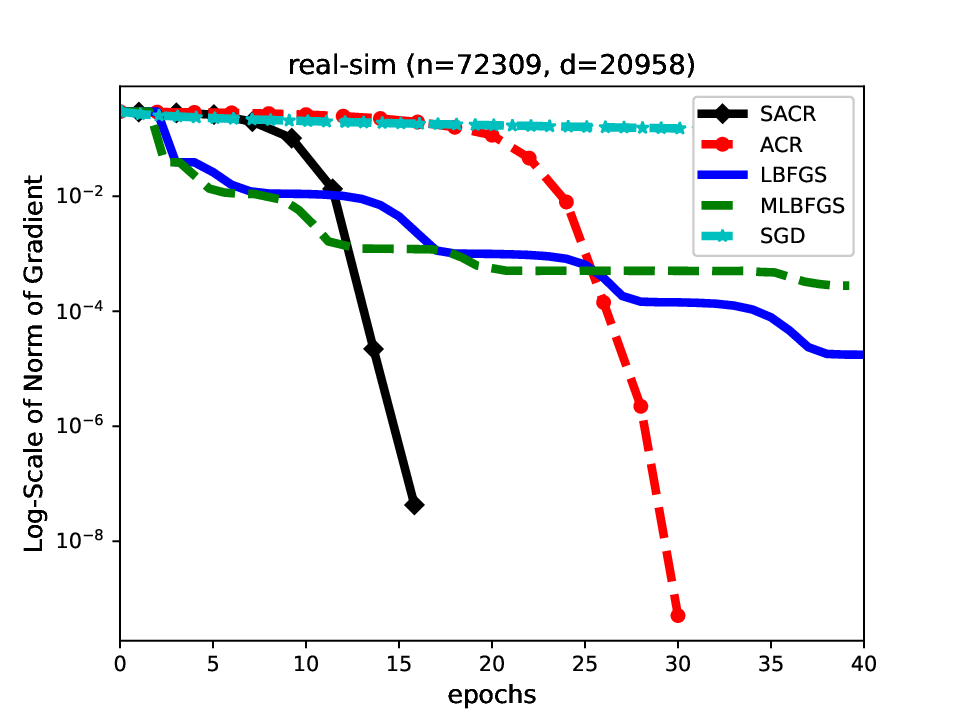}
\caption{Performance of our algorithm and four state-of-the-art {algorithms without sub-sampled Hessian information} on eight datasets with the log-scale of the norm of gradient vs.\ number of epochs. }\label{fig:result-low-epoch}
\end{figure*}
\paragraph{Problem.} Given a collection of data samples $\{\left(\w_i, y_i\right)\}_{i=1}^n$ in which $y_i \in \{-1, 1\}$, the model of regularized logistic regression is given by
\begin{equation}\label{prob:logistic-regression}
\min_{\x\in\br^d} \ \frac{1}{n}\sum_{i=1}^n \log\left(1+e^{-y_i \w_i^\top\x}\right) + \left(\frac{\lambda}{2}\right)\left\|\x\right\|_2.
\end{equation}
where the regularization term $\left\|\cdot\right\|_2$ promotes smoothness and $\lambda>0$ balances smoothness with goodness-of-fit and generalization and is chosen by five-fold cross validation.
\begin{figure*}[!ht]
\centering
\includegraphics[width=0.40\textwidth]{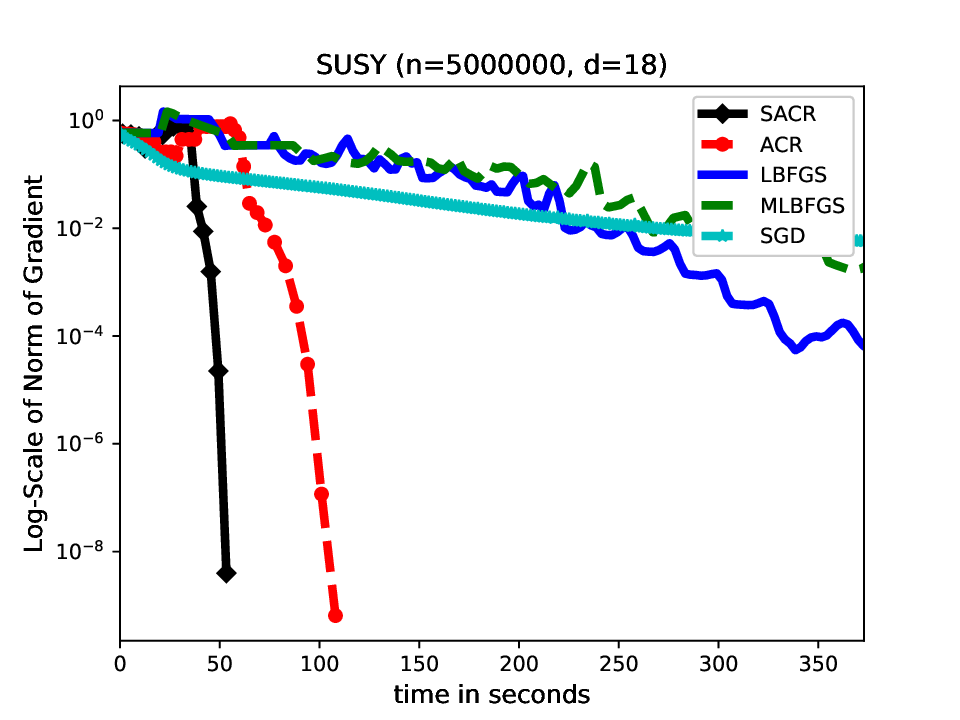}
\includegraphics[width=0.40\textwidth]{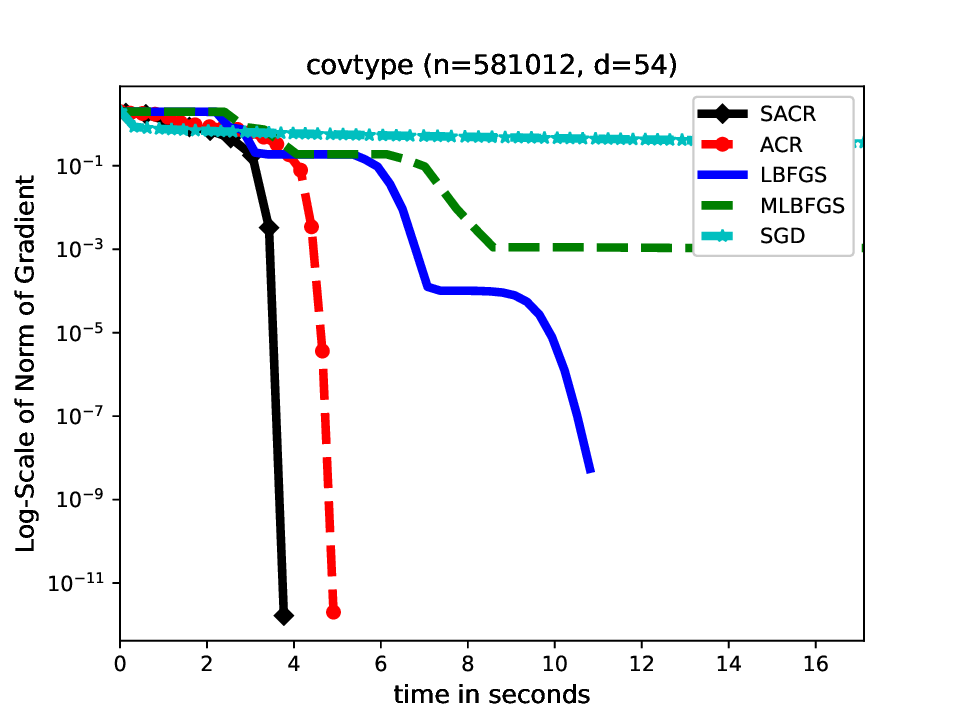}
\includegraphics[width=0.40\textwidth]{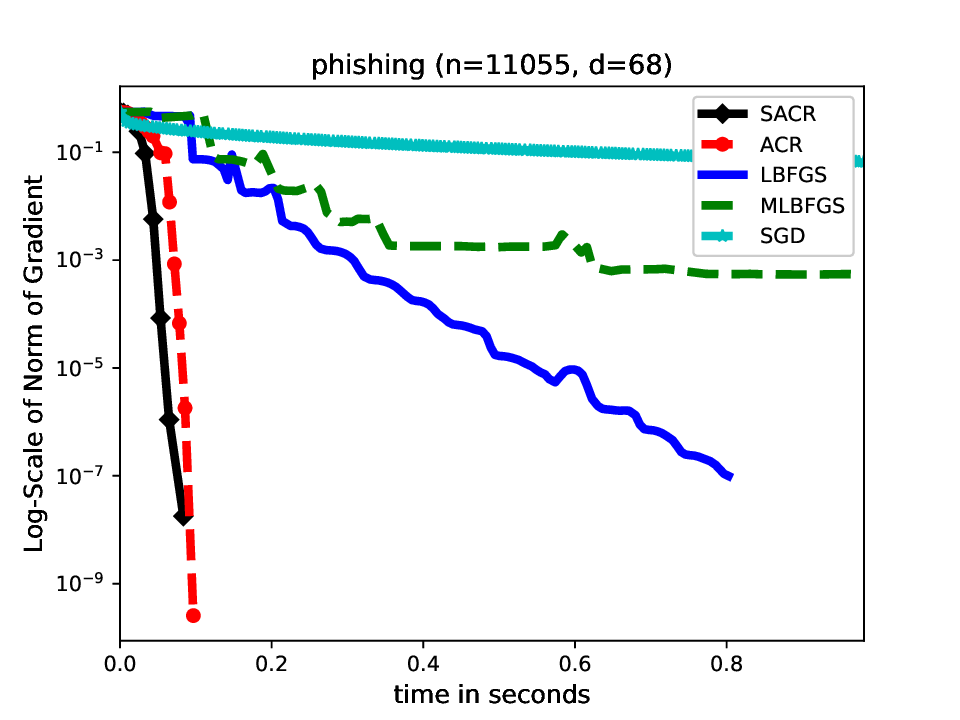}
\includegraphics[width=0.40\textwidth]{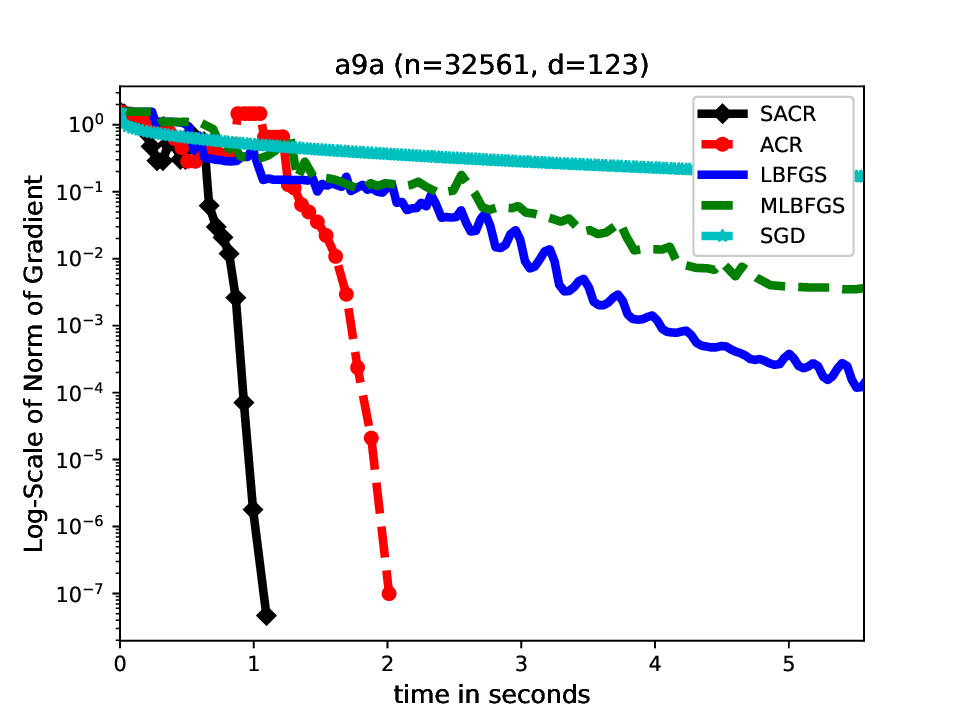}
\includegraphics[width=0.40\textwidth]{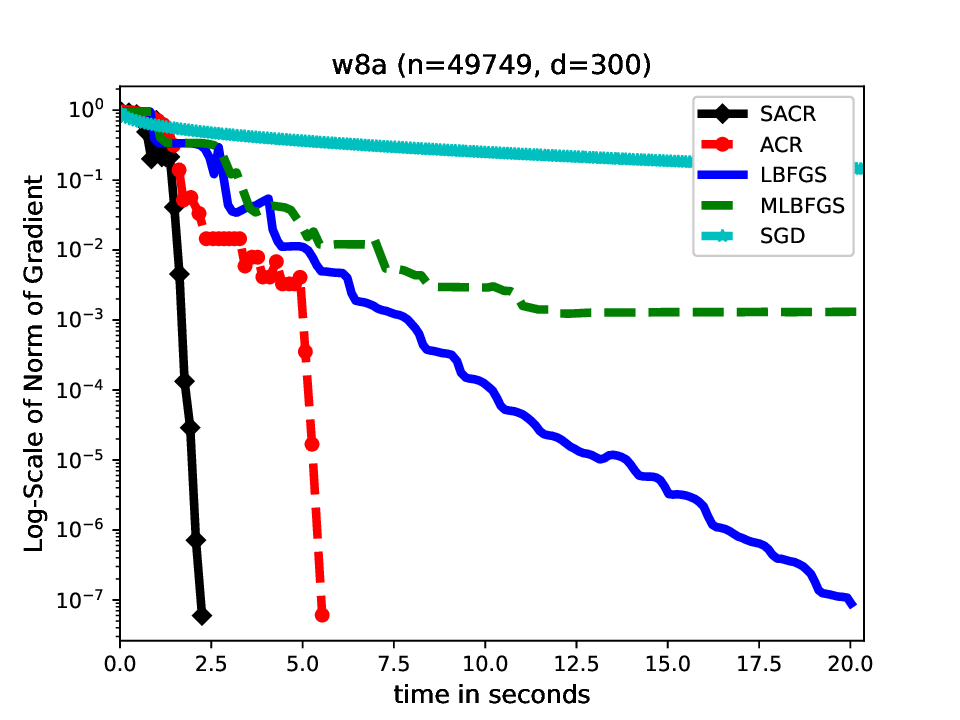}
\includegraphics[width=0.40\textwidth]{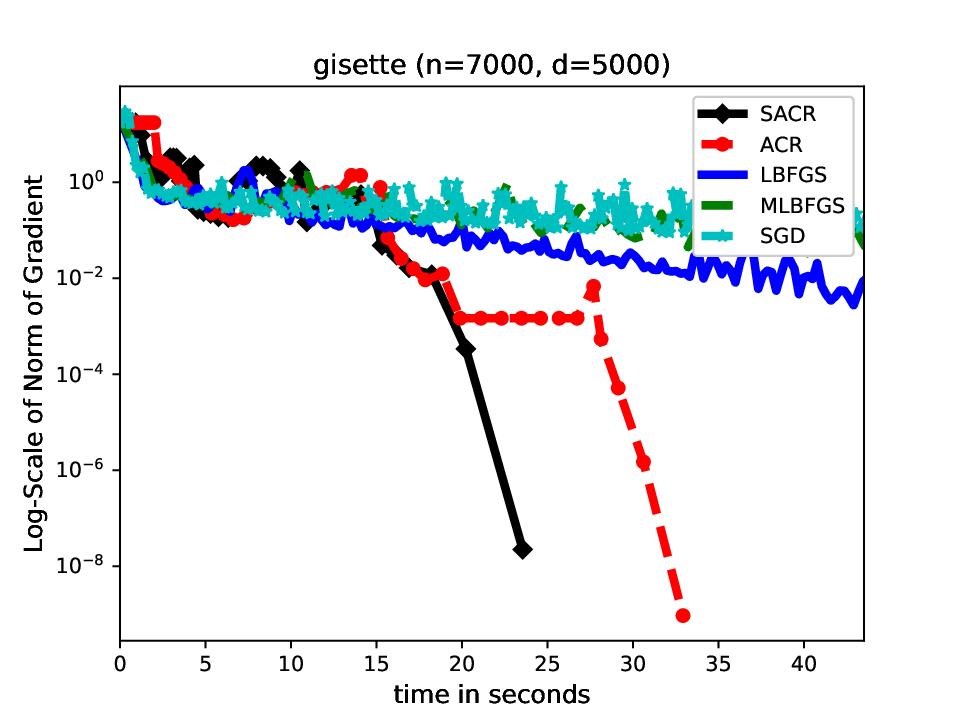}
\includegraphics[width=0.40\textwidth]{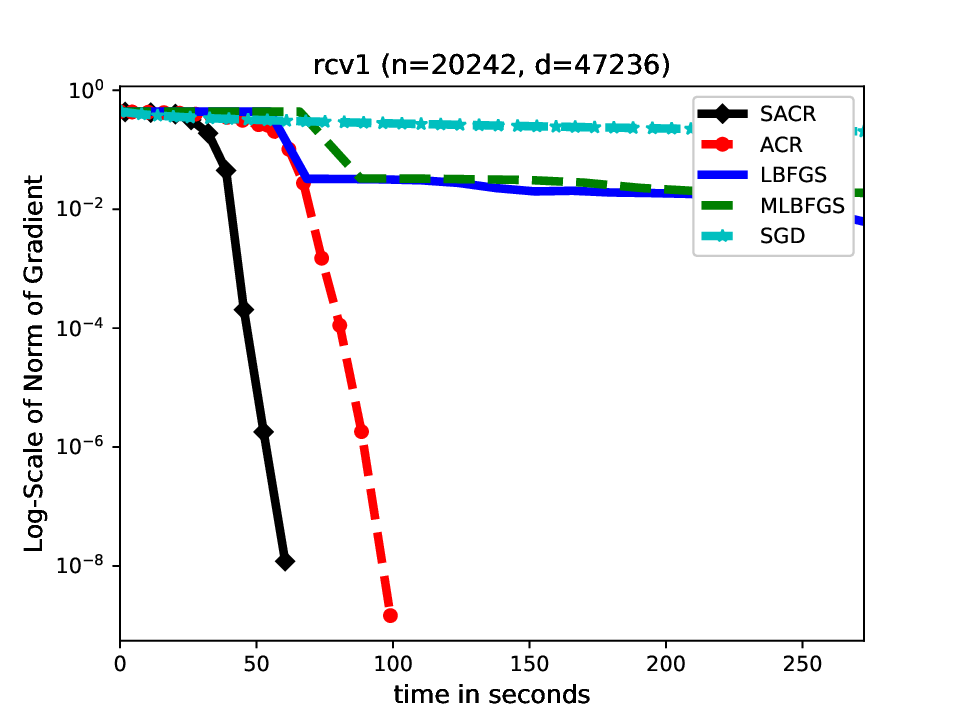}
\includegraphics[width=0.40\textwidth]{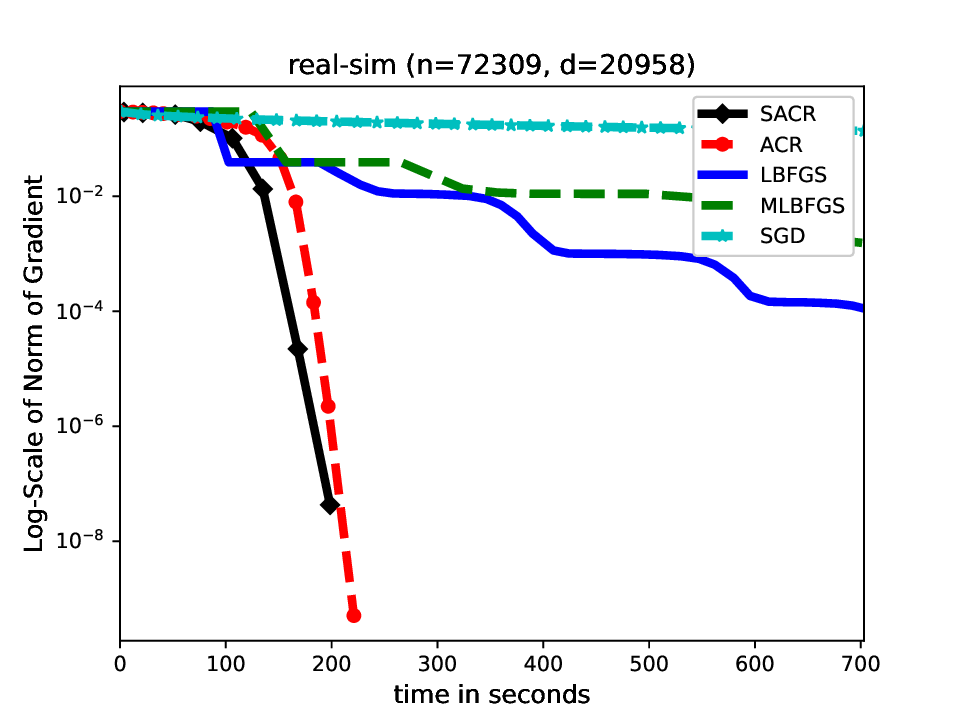}
\caption{Performance of our algorithm and four state-of-the-art {algorithms without sub-sampled Hessian information} on eight datasets with the log-scale of the norm of gradient vs.\ time. }\label{fig:result-low-time}
\end{figure*}
\paragraph{Experimental setting.} We implement Algorithm~\ref{Algorithm:SAARC} with $\eta = 0.1$, {$\gamma_1 =\gamma_2 = \gamma_3=2$}, $\sigma_{\min} = 10^{-16}$, $\sigma_0 = 1$ and $\kappa_\theta = 0.1$, denoted as {SACR}, in a hybrid manner. Specifically, given that {SACR} contains two phases we implement {SACR} with these two phases at the beginning and stop the second phase when the iterate is relatively close to the optimal solution, and then switch to subsampled cubic regularization {(SCR)} method. This is because that we observe that the first phase mainly contributes to the local convergence of {SACR} while the second phase may hurt it. In fact, when the iterate is close enough {to an optimal solution}, the first phase reduces to the Newton method, hence admitting a local quadratic convergence rate. In our experiment, we stop the second phase when $|f(\x_{i+1})-f(\x_i)|/|f(\x_i)| \leq 10^{-1}$ and the final stopping criterion as $\|\nabla f(x)\| \leq 10^{-7}$.

Furthermore, since {the accelerated convergence of our algorithm is global, to observe the effect of acceleration,} we need to set the initial solution far {away} from the local convergence region. In this case, we randomly generate the starting point from a Gaussian random variable with zero mean and a large variance. {The sample size is chosen inversely proportional to the square norm of the gradient (cf.\ Lemmas~\ref{Lemma:Uniform-Sample} and \ref{Lemma:NonUniform-Sample}) {with proportional ratio $0.2\log(100d)$} and are bounded both below and above by some constants. The lower bound and the upper bound are tuned for different datasets. In addition, we require the lower bound of the sample size decreases to a smaller value after we switching SACR to SCR in the local region of the optimal solution. 
} 

Finally, we use the log-scale of norm of gradient as the function of number of epochs and run time as the metric in our experiment. In particular, one epoch is counted when a full batch size (i.e.\ $n$ times) of the gradient or Hessian of the component functions is queried. Since the sample size in sub-sampling algorithms is less than $n$, one epoch is likely to be consumed by the queries from several iterations.

\paragraph{Subproblem solving.} The generalized conjugate gradient method with Lanczos process is applied to approximately solve the cubic regularized subproblem. More specifically, the convex cubic polynomial in the subproblem is minimized over a Krylov subspace, which is defined by the gradient and Hessian of $f$ at $\x_i$ and given by
\begin{equation*}
\KCal := \textsf{Span}\left\{\nabla f(\x_i), \ \nabla^2 f(\x_i) \nabla f(\x_i), \ \left( \nabla^2f(\x_i) \right)^2 \nabla f(\x_i), \  \ldots\right\},
\end{equation*}
Note that the Krylov subspace $\KCal$ gradually swells with very cheap computational cost since each orthogonal basis is created by performing a single matrix-vector product. Furthermore, the minimization of cubic polynomial over the Krylov subspace only requires factorizing a tri-diagonal matrix at the $O(d)$ expense. Finally, we set the stopping criterion for subproblem solving as~\eqref{Eqn:Approx_Subprob_SAARC} with $\kappa_\theta = 0.1$.

\subsection{Comparision with four state-of-the-art {algorithms without sub-sampled Hessian information}}
In the first experiment, we compare our algorithm to four baseline algorithms, including the deterministic counterpart of our algorithm~\cite{Jiang-2020-Unified}, denoted as ACR, the limited memory BFGS method, denoted as LBFGS, the minibatch variant of LBFGS with the batch size $n/2$, denoted as MLBFGS, and the minibatch variant of stochastic gradient descent with the batch size $n/10$, denoted as SGD. {Note that only the gradient is sampled in MLBFGS and SGD, while the Hessian instead of the gradient is subsampled in our algorithm.} For LBFGS and MLBFGS implementations, we set an initial matrix as identity matrix and the line search criterion with strong Wolfe condition. The ratio for measuring the progress is set as $0.9$, the maximum number of line search is set as $5$ and the memory size is set as $30$. Additionally, we excluded the sub-sampled Newton method since its global convergence is unknown in general and, when the iterate is close enough, our algorithm turns out to be the same as {the} sub-sampled Newton method since the cubic regularization term will become very small. 

The results on eight datasets are presented in Figures~\ref{fig:result-low-epoch}-\ref{fig:result-low-time}. We observe that SACR outperforms other algorithms in most of the datasets despite the competitive performance of other algorithms at the initial stage. In particular, both SACR and ACR can attain the solution with high accuracy while LBFGS, MLBFGS and SGD can not. {We observe the curve of LBFGS lies between that of SGD and ACR, and it behaves more like ACR for low-dimensional dataset (i.e., SUSY and covtype). This is probably due 
	to that LBFGS can be viewed as an interplay between the first order and the second order method, and it exhibits superlinear convergence thanks to certain geometrical regularity (e.g., restricted strongly convexity) that intuitively exists for low-dimensional problem with high probability in real applications.}
Also, SACR is more efficient than {ACR} due to the usage of sub-sampling techniques. When the dimension of the dataset becomes larger, standard second-order methods suffer from the storing and computing the inverse of the Hessian as the dimension increases, while our algorithm remains efficient in most of these datasets. This is not surprising since the subproblem solving depends on the generalized conjugate gradient method. For high-dimensional problems, storing the Hessian appears to be a critical issue which requires further exploration. To this end, the competitive performance demonstrates that our algorithm has a great potential to achieve practical performance on the large-scale problems.

{\subsection{Comparision with three types of sub-sampled cubic regularized algorithms}
In the second experiment, we compare our algorithm to three types of sub-sampled cubic regularized algorithms including the non-accelerated sub-sampled cubic regularized (SCR) algorithm that is used in the phase I of SACR, a variant of SCR in \citep{Kohler-2017-SubSample} denoted as SCR-KL, a dynamic inexact Hessian variant of SCR in \citep{Bellavia-2021-Finite-Sum} denoted as SCR-BGM. In the implementation, the the sample size for the approximation in SCR-KL is determined by the previous stepsize instead of the current stepsize used in the theory \citep{Kohler-2017-SubSample} {with an adaptive rule\footnote{\textsc{https://github.com/dalab/subsampled{\_}cubic{\_}regularization}}}. While the parameters of SCR-BGM strictly follows the setting for testing the real datasets in \citep{Bellavia-2021-Finite-Sum}. In addition,
we impose a lower bound as well as an upper bound of the sample size for all the algorithms, and these bounds are tuned for different datasets. In particular, {the lower bound is uniformly set to be $0.01\cdot n$ for all datasets, while the upper bound is set to be $0.2\cdot n$ for $7$ datasets except the full batch size $n$ is used for the dataset ``gisette''.}

We provide the corresponding numerical results on eight datasets in Figures~\ref{fig:subsample-alg-result-low-epoch}, where the log-scale of the norm of gradient vs.\ number of epochs is provided. We can see that SACR outperforms other sub-sampled cubic regularized algorithms in all the datasets, while all the tested algorithms have similar convergence behavior after the iteration points  entering the local region of the optimal solution. This indeed indicates that our  technique really accelerates the algorithm in 
 finding such local region and yields a faster convergence rate.  

\begin{figure*}[!ht]
	\centering
	\includegraphics[width=0.40\textwidth]{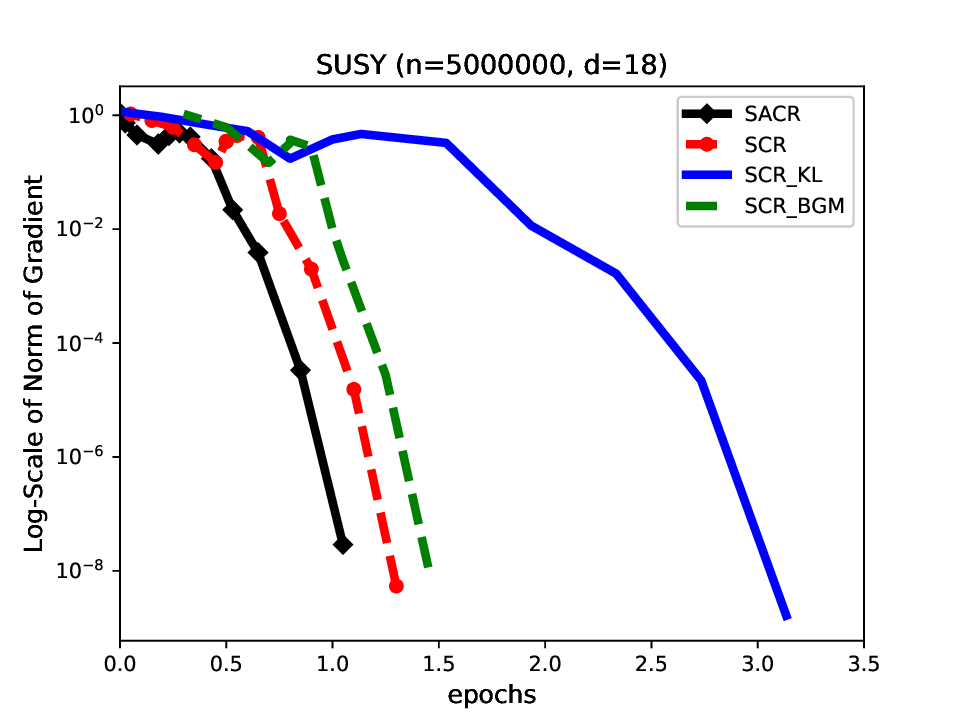}
	\includegraphics[width=0.40\textwidth]{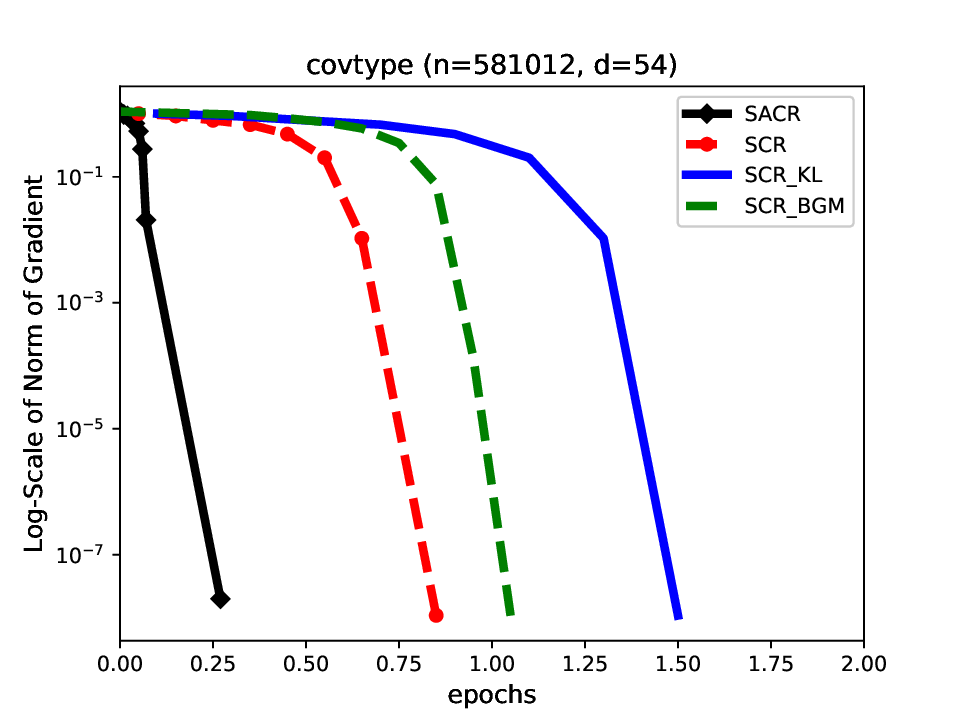}
	\includegraphics[width=0.40\textwidth]{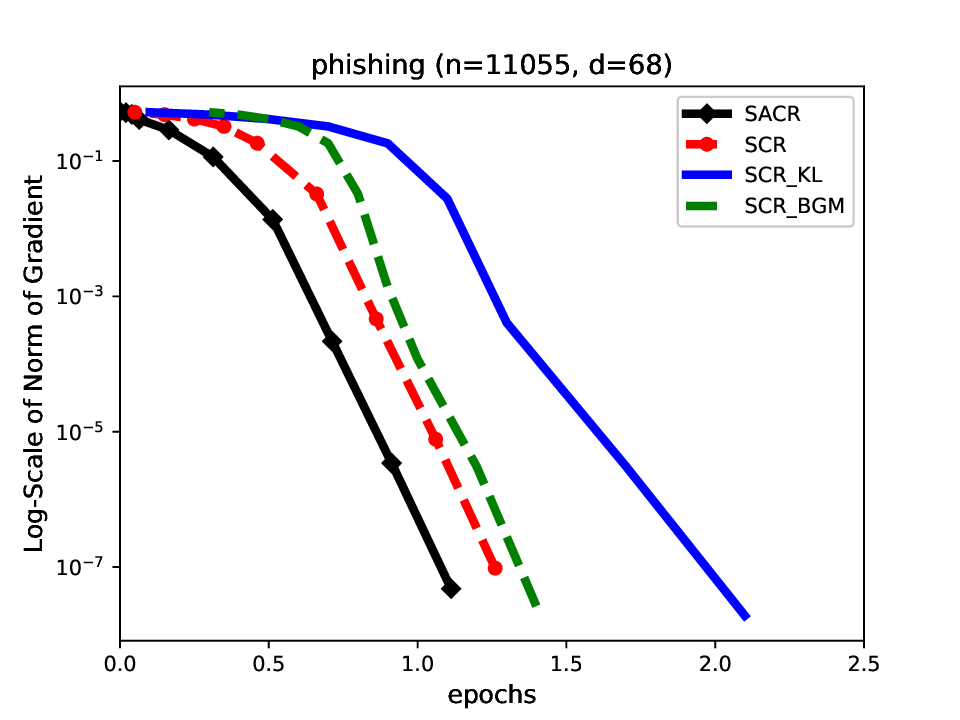}
	\includegraphics[width=0.40\textwidth]{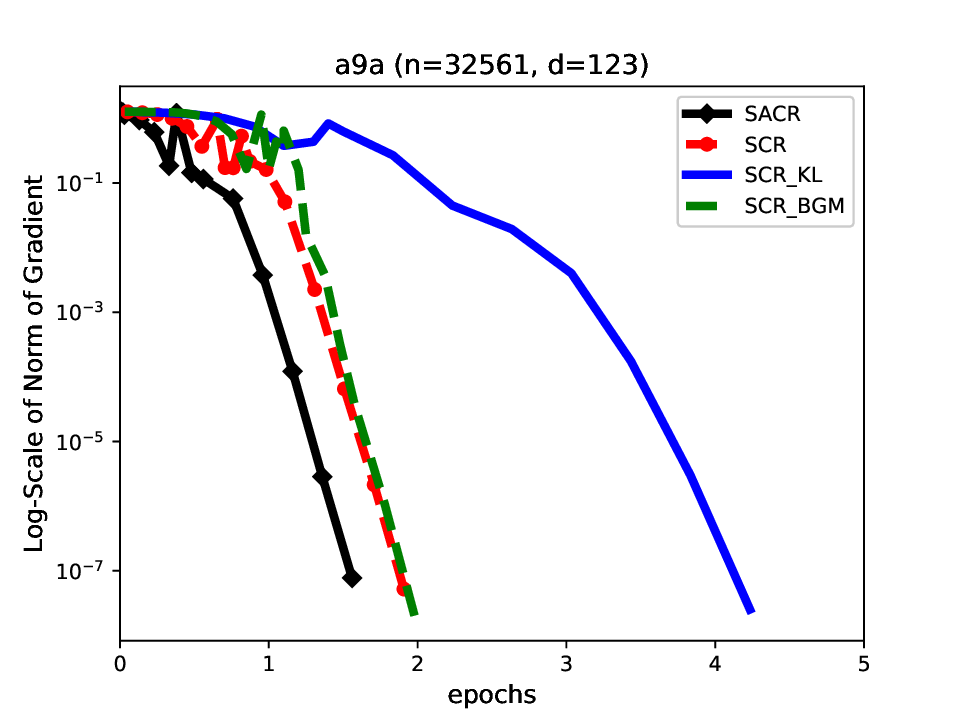}
	\includegraphics[width=0.40\textwidth]{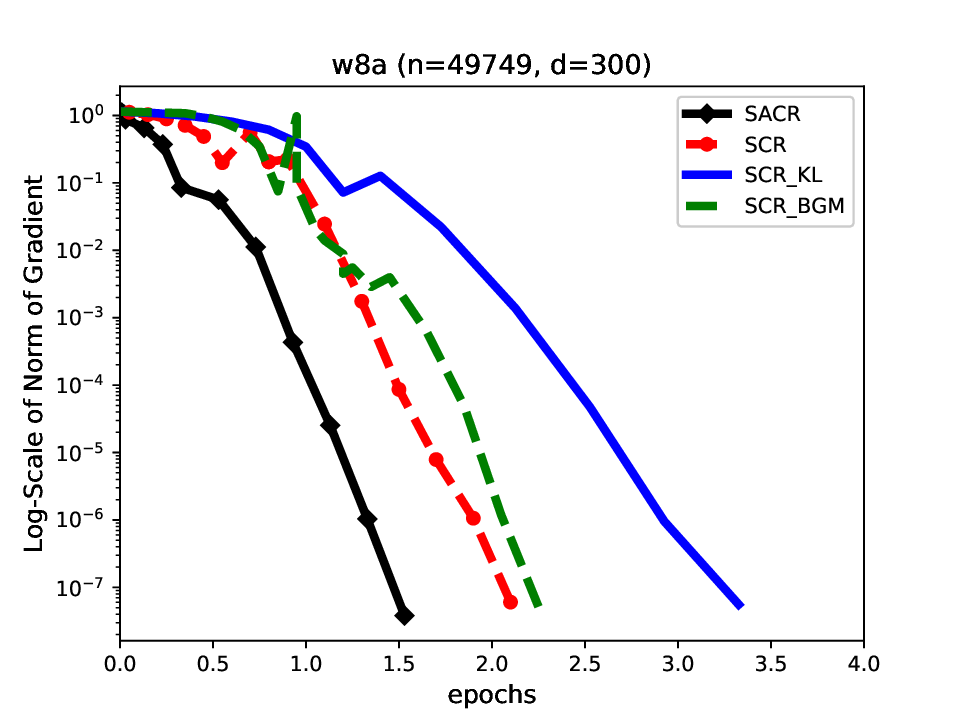}
	\includegraphics[width=0.40\textwidth]{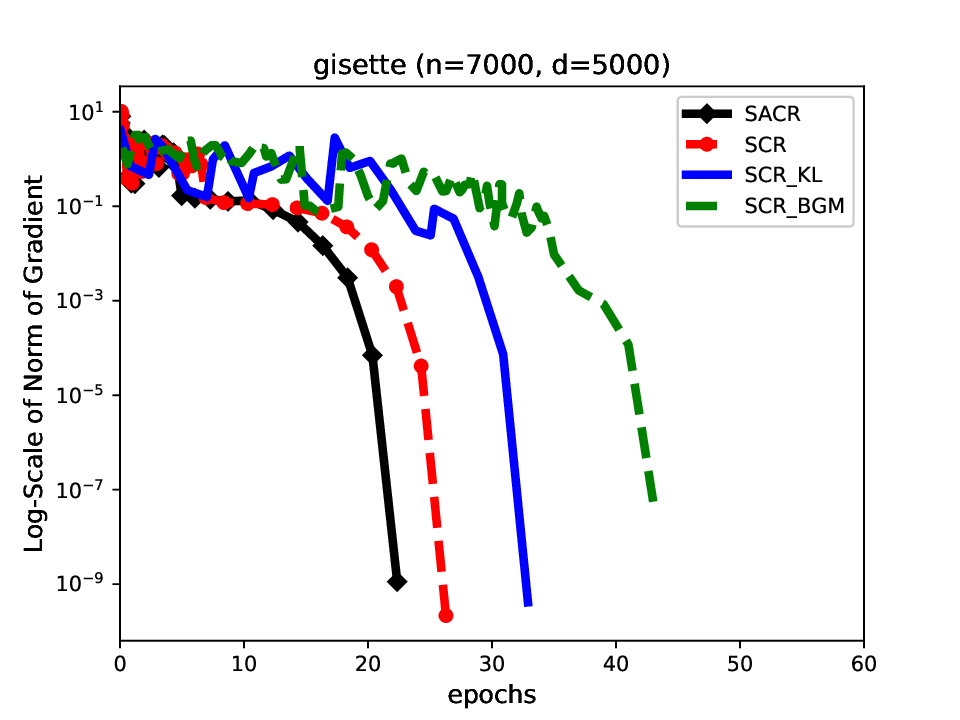}
	\includegraphics[width=0.40\textwidth]{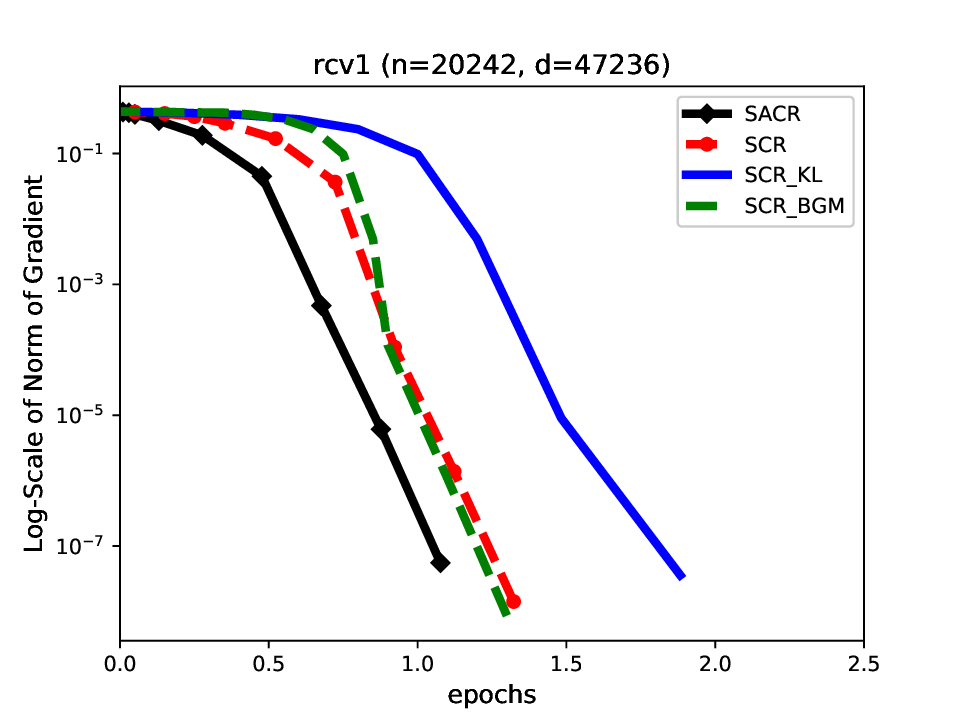}
	\includegraphics[width=0.40\textwidth]{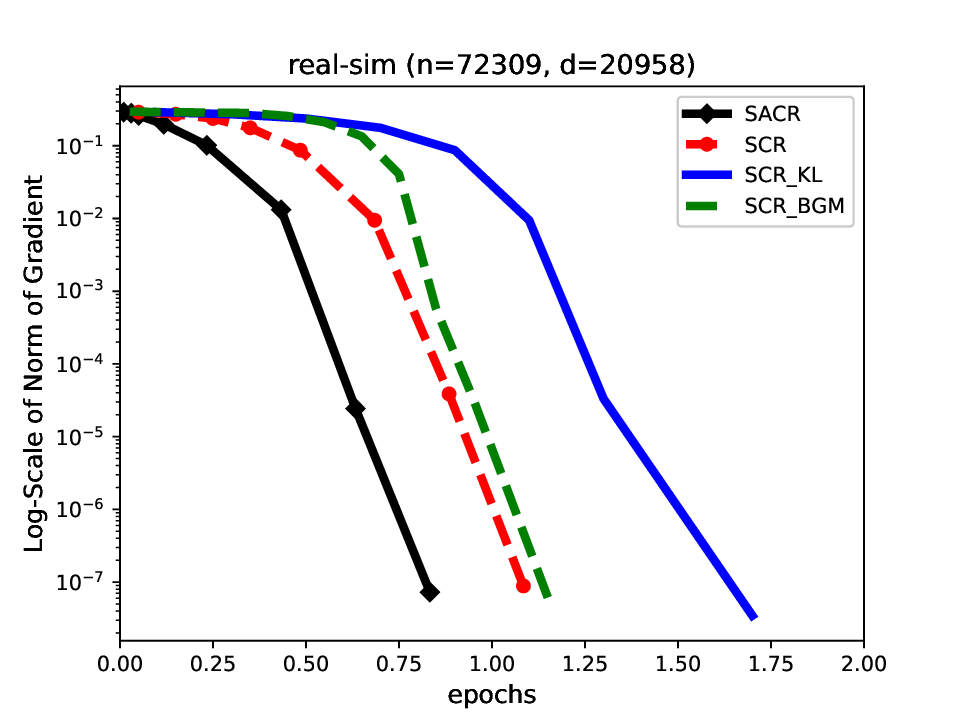}
	\caption{Performance of our algorithm and three sub-sampled cubic regularized algorithms on eight datasets with the log-scale of the norm of gradient vs.\ number of epochs. }\label{fig:subsample-alg-result-low-epoch}
\end{figure*}
}

\section{Concluding Remarks}
The theoretical properties of subsampled Hessian Newton-type methods have recently received a lot of  attention, but their acceleration has not been well studied  in the literature. In this paper, we focus on the sum-of-nonconvex problem and propose a novel way to accelerate adaptive cubic regularization of Newton's method with either \textit{uniform} or \textit{non-uniform} subsampled Hessians. Our new algorithm achieves the global iteration complexity of $O(\epsilon^{-1/3})$ with high probability, which matches that of the original accelerated cubic regularization methods~\citep{Jiang-2020-Unified} using the \textit{full}\/ Hessian information. In the worst case scenario, we demonstrate that our algorithm still achieves an $O(\epsilon^{-5/6}\log(\epsilon^{-1}))$ iteration complexity bound. The proof techniques are new to our knowledge and can be of independent interets.  

Our empirical evaluation show that the new algorithm studied in this paper is generally more efficient than its deterministic counterpart, LBFGS, mini-batch LBFGS, and SGD, on regularized logistic regression problems for real datasets.

\bibliographystyle{plainnat}
\bibliography{ref}

\newpage

\appendix
\section{Proofs in Section \ref{Section:Probabilistic}}\label{Sec: Prob-Proofs}
We prove Lemma \ref{Lemma:SAARC-T1} and Lemma \ref{Lemma:SAARC-T2}, which describe the relation between the total iteration numbers in Algorithm \ref{Alg:SSAS} and the amount of successful iterations $|\SucCal|$ in Algorithm \ref{Alg:ASAS}.

\paragraph{Proof of Lemma \ref{Lemma:SAARC-T1}:}
First by invoking the fact
\begin{equation}\small
f(\x_i + \s_i)=f(\x_i) + \s_i^\top \nabla f(\x_i) + \frac{1}{2} \s_i^\top \nabla^2 f(\x_i) \s_i + \int_{0}^{1} (1 - \tau) \s_i^\top[\nabla^2 f(\x_i + \tau \s_i) - \nabla^2 f(\x_i)] \s_i \; d \tau,  \label{Tailor-Expansion}
\end{equation}
we have that
\begin{eqnarray}
& & f(\x_i + \s_i) \nonumber \\
& = & f(\x_i) + \s_i^\top \nabla f(\x_i) + \frac{1}{2} \s_i^\top \nabla^2 f(\x_i) \s_i + \int_{0}^{1} (1 - \tau) \s_i^\top[\nabla^2 f(\x_i + \tau \s_i) - \nabla^2 f(\x_i)] \s_i \; d \tau  \nonumber \\
& {\eqref{Def:Lipschitz-Hessian} \above 0pt \leq} & f(\x_i) + \s_i^\top \nabla f(\x_i) + \frac{1}{2} \s_i^\top \nabla^2 f(\x_i) \s_i  + \frac{\bar{\rho}}{6}\|\s_i\|^3  \nonumber \\
& = & m(\s_i; \x_i, \sigma_i) +  \frac{1}{2} \s_i^\top\left(\nabla^2 f(\x_i) - H(\x_i)\right) \s_i + \left(\frac{\bar{\rho}}{6}-\frac{\sigma_i}{3}\right)\|\s_i\|^3 \nonumber \\
& {\eqref{Hession-Approximation} \above 0pt \leq} & m(\s_i; \x_i, \sigma_i) + {\epsilon_i}\|\s_i\|^2 + \left(\frac{\bar{\rho}}{6} - \frac{\sigma_i}{3}\right)\|\s_i\|^3. \nonumber
\end{eqnarray}
Next we argue that when $\sigma_i$ exceeds a certain constant, then it holds that
\begin{equation*}\label{f-less-m}
f(\x_i+\s_i)\leq m(\s_i; \x_i, \sigma_i).
\end{equation*}
The analysis is conducted according to the value of  $\left\|\s_i\right\|$ in two cases.
\begin{enumerate}
\item When $\left\|\s_i\right\|\geq 1$, we have
\begin{equation*}
f(\x_i + \s_i) \leq m(\s_i; \x_i, \sigma_i) + \left({\epsilon_i} + \frac{\bar{\rho}}{6} - \frac{\sigma_i}{3}\right)\|\s_i\|^3 ,
\end{equation*}
which in combination with the fact that $\epsilon_i \leq \epsilon_0 \le 1$ leads to
\begin{equation*}
\sigma_i\geq\frac{6 + \bar{\rho}}{2} \quad \Longrightarrow \quad f(\x_i+\s_i)\leq m(\s_i; \x_i, \sigma_i).
\end{equation*}
\item When $\|\s_i\|<1$, { 
 according to Condition~\ref{Cond:Approx_Subprob_SAARC}, 
}
it holds that
\begin{eqnarray}
\kappa_\theta\|\s_i\| > \kappa_\theta\|\s_i\|^2
& { \ge } & \|\nabla f(\x_i) + H(\x_i) \s_i + \sigma_i\|\s_i\|\cdot\s_i\| \nonumber\\
& \geq & \|\nabla f(\x_i)\| - \| H(\x_i)\|\|\s_i\| - \sigma_i\|\s_i\|^2 \nonumber \\
& \geq & \|\nabla f(\x_i)\| - (L+{ \epsilon_0}+\sigma_i)\|\s_i\| \nonumber ,
\end{eqnarray}
where the last inequality holds true since $\left\|\s_i\right\|<1$ and {
\begin{equation*}\label{bounded-H-x}
\| H(\x_i)\| \le \frac{1}{n|\SCal|}\sum_{j\in\SCal} \frac{1}{p_j} \| \nabla^2 f_j(\x_i) \| + \epsilon_i \|{\BI}\| \overset{\eqref{Bounded-Hessian}}\le \frac{1}{n|\SCal|}\sum_{j\in\SCal} \frac{1}{p_j} L + \epsilon_i  \le L + \epsilon_0.
\end{equation*}
 This can be further rewritten as}
\begin{equation}\label{lowerbound-si}
\|\s_i\| \geq \frac{\|\nabla f(\x_i)\|}{L+{ \epsilon_0} +\sigma_i + \kappa_\theta}.
\end{equation}
Moreover, recall that
\begin{equation*}
f(\x_i + \s_i) \leq m(\s_i; \x_i, \sigma_i) + \left(\frac{\epsilon_i}{\|\s_i\|} + \frac{\bar{\rho}}{6} - \frac{\sigma_i}{3}\right)\|\s_i\|^3 ,
\end{equation*}
and combining the above two inequalities yields that
\begin{equation*}
{\frac{\epsilon_i(L+{ \epsilon_0}+\sigma_i + \kappa_\theta)}{ \|\nabla f(\x_i)\|} + \frac{\bar{\rho}}{6} - \frac{\sigma_i}{3} \le 0 \Longrightarrow f(\x_i+\s_i) \leq m(\s_i; \x_i, \sigma_i).} \end{equation*}
Recall in Algorithm \ref{Alg:SSAS} that
\begin{equation}\label{upperbound-epsiloni}
\epsilon_i = \min\left\{\frac{\|\nabla f(\x_i)\|}{ 6}, \ {  \epsilon_{0}}\right\} \leq \frac{\|\nabla f(\x_i)\|}{ 6},
\end{equation}
then it suffices to show
\begin{equation*}
\frac{L+{ \epsilon_0}+\sigma_i+ \kappa_\theta }{6} + \frac{\bar{\rho}}{6} - \frac{\sigma_i}{3} \leq 0.
\end{equation*}
That is:
\begin{equation*}
\sigma_i\geq L+{ \epsilon_0}+ \kappa_\theta   +  \bar \rho    \Longrightarrow f(\x_i+\s_i) \leq m(\s_i; \x_i, \sigma_i).
\end{equation*}
\end{enumerate}
In summary, we have concluded that
\begin{equation}\label{sigma-imply}
\sigma_i\geq  \max\left\{\frac{6 + \bar\rho}{2}, \  L+{ \epsilon_0}+ \kappa_\theta   +  \bar \rho      \right\} \Longrightarrow f(\x_i+\s_i) \leq m(\s_i; \x_i, \sigma_i),
\end{equation}
which implies that $\sigma_i<\max\{3 + 0.5\bar\rho, \   L+{ \epsilon_0}+ \kappa_\theta   +  \bar \rho      \}$ for $i\le T_1-2$. Moreover,
\begin{equation*}
{\sigma_{T_1} = \sigma_{T_1 -1} \leq \gamma_2  \sigma_{T_1 -2} \leq \gamma_2 \max\left\{3 + 0.5\bar\rho, \   L+{ \epsilon_0}+ \kappa_\theta   +  \bar \rho      \right\}.}
\end{equation*}
Then it holds that $\sigma_i \le \bar{\sigma}^P_1 = { \max\{\sigma_0, \ 3\gamma_2 + 0.5\bar\rho \gamma_2, \  \gamma_2  (L+\epsilon_0+ \kappa_\theta +   \bar \rho   )    \} }$ for any $i \le T_1$. On the other hand, it follows from the construction of Algorithm \ref{Alg:SSAS} that $\sigma_{\min}\leq\sigma_i$ for all iterations, and $\gamma_1\sigma_i\leq\sigma_{i+1}$ for all unsuccessful iterations. Consequently, we have
{
\begin{equation*}
\frac{\bar{\sigma}^P_1}{\sigma_{\min}} \geq \frac{\sigma_{T_1}}{\sigma_0} = \frac{\sigma_{T_1}}{\sigma_{T_1-1}} \cdot \prod_{j=0}^{T_1-2} \frac{\sigma_{j+1}}{\sigma_j} =  \prod_{j=0}^{T_1-2} \frac{\sigma_{j+1}}{\sigma_j} \geq \gamma_1^{T_1-1},
\end{equation*}
where the second equality is due to $\sigma_{T_1} = \sigma_{T_1-1}$ in Algorithm \ref{Alg:SSAS},
}
and hence
\begin{equation*}
T_1 \leq \left(1+ \frac{{1}}{\log\gamma_1} \log\left(\frac{\bar{\sigma}^P_1}{\sigma_{\min}}\right)\right).
\end{equation*}
This completes the proof of Lemma \ref{Lemma:SAARC-T1}. {\hfill $\Box$ \vskip 0.4cm}
\paragraph{Proof of Lemma \ref{Lemma:SAARC-T2}:}
We have
\begin{eqnarray*}
& & \s_{j}^\top \nabla f(\y_l+\s_{j})\\
& = & \s_{j}^\top[\nabla f(\y_l + \s_{j}) - \nabla f(\y_l) - \nabla^2 f(\y_l) \s_{j}] + \s_{j}^\top[\nabla f(\y_l) + \nabla^2 f(\y_l) \s_{j}] \\
& \leq & \|\nabla f(\y_l+\s_{j}) - \nabla f(\y_l) - \nabla^2 f(\y_l) \s_{j}\|\|\s_{j}\| + \s_{j}^\top[\nabla f(\y_l) + H(\y_l) \s_{j} + \sigma_j \|\s_j \| \s_j] \\
& & + \s_{j}^\top(\nabla^2 f(\y_l) - H(\y_l))\s_j - \sigma_j\|\s_j\|^3 \\
& {\eqref{Eqn:Approx_Subprob_SAARC} \above 0pt  \leq }  & \|\nabla f(\y_l+\s_j) - \nabla f(\y_l) - \nabla^2 f(\y_l) \s_{j}\|\|\s_j\| + (\kappa_\theta - \sigma_j)\|\s_j\|^3 + 2\epsilon_j\|\s_j\|^2 \\
& = & \|\int_0^1 [\nabla^2 f(\y_l + \tau \cdot \s_j) - \nabla^2 f(\y_l)] \s_j \; d \tau\| \|\s_j\| + (\kappa_\theta - \sigma_j)\|\s_j\|^3 + 2\epsilon_j\|\s_j\|^2 \\
& \eqref{Def:Lipschitz-Hessian} \above 0pt  \leq & \left(\frac{\bar{\rho}}{2} + \kappa_\theta - \sigma_{j}\right)\|\s_j\|^3 + 2\epsilon_j\|\s_j\|^2.
\end{eqnarray*}
Next we argue that when $\sigma_i$ exceeds certain constant, it holds
\begin{equation*}\label{sf-more-eta}
-\frac{\s_{j}^\top \nabla f(\y_l+\s_{j})}{\|\s_j\|^3} \geq \eta,
\end{equation*}
The analysis is conducted according to the value of $\|\s_j\|$ in two cases.
\begin{enumerate}
\item When $\|\s_j\|\geq 1$, we have
\begin{equation*}
\s_j^\top \nabla f(\y_l+\s_j) \leq \left(\frac{\bar{\rho}}{2} + \kappa_\theta - \sigma_j + 2\epsilon_j\right)\|\s_j\|^3,
\end{equation*}
which combined with $\epsilon_j \leq \epsilon_0 \le 1$ implies that
\begin{equation*}
\sigma_{j} \geq \frac{\bar{\rho}}{2} + \kappa_\theta+\eta+2 \Longrightarrow -\frac{\s_j^\top \nabla f(\y_l+\s_j) }{\|\s_j\|^3} \geq \eta.
\end{equation*}
\item When $\|\s_j\|<1$, similar argument of \eqref{lowerbound-si} implies that
\begin{equation}\label{lowerbound2-si}
\|\s_{j}\| \geq \frac{\|\nabla f(\y_l)\|}{L+{ \epsilon_0}+ \sigma_i + \kappa_\theta}.
\end{equation}
Moreover, recall that
\begin{equation*}
\s_j^\top  \nabla f(\y_l+\s_{j}) \leq \left(\frac{\bar{\rho}}{2} + \kappa_\theta - \sigma_{j}+\frac{2\epsilon_{j}}{\|\s_j\|}\right)\|\s_{j}\|^3.
\end{equation*}
Combining the above two inequalities yields that
\begin{equation*}
\frac{2\epsilon_j(L+{ \epsilon_0}+ \sigma_i + \kappa_\theta)}{\|\nabla f(\y_l)\|} + \frac{\bar{\rho}}{2} + \kappa_\theta - \sigma_{j} +\eta \leq 0 \Longrightarrow -\frac{\s_j^\top \nabla f(\y_l+\s_{j})}{\|\s_{j}\|^3} \geq \eta.
\end{equation*}
Recall {in Algorithm \ref{Alg:ASAS}} that
\begin{equation*}
\epsilon_{j} = \min\left\{ \frac{\|\nabla f(\y_l)\|}{4} , \ {  \epsilon_{0}} \right\} \leq \frac{\|\nabla f(\y_l)\|}{4},
\end{equation*}
then it suffices to show
\begin{equation*}
\frac{L+{ \epsilon_0}+ \sigma_{j}+ \kappa_\theta}{2} + \frac{\bar{\rho}}{2} + \kappa_\theta - \sigma_{j} + \eta \leq 0.
\end{equation*}
That is,
\begin{equation*}
\sigma_{j} \geq L+{ \epsilon_0}+ \bar\rho + 3\kappa_\theta + 2\eta \Longrightarrow -\frac{\s_{j}^\top  \nabla f(\y_l+\s_{j})}{\|\s_{j}\|^3} \geq \eta.
\end{equation*}
\end{enumerate}
In summary, we have concluded that
\begin{equation*}
\sigma_{j} \geq \max\left\{\frac{\bar{\rho}}{2} + \kappa_\theta+\eta+2, L+{ \epsilon_0}+ \bar\rho + 3\kappa_\theta + 2\eta \right\} \Longrightarrow -\frac{\s_{j}^\top\nabla f(\y_l+\s_{j})}{\|\s_{j}\|^3} \geq \eta ,
\end{equation*}
which further implies that for any unsuccessful iteration $j \not\in \SucCal$, the following inequality holds true, 
\begin{equation*}
\sigma_{j} < \max\left\{\frac{\bar{\rho}}{2} + \kappa_\theta+\eta+2, \ L+{ \epsilon_0}+ \bar\rho + 3\kappa_\theta +2\eta\right\}.
\end{equation*}
Therefore, for any successful iteration $j \in \SucCal$, we have
\begin{equation*}
\sigma_{j+1}\leq \sigma_{j}\leq \gamma_2 \cdot \sigma_{j-1}\leq \gamma_2 \max\left\{\frac{\bar{\rho}}{2} + \kappa_\theta+\eta+2, \ L+{ \epsilon_0}+ \bar\rho + 3\kappa_\theta +2\eta\right\}.
\end{equation*}
Consequently, for any $0\le j \le T_2$, we have
\begin{equation}\label{bound-sigma-j}
\sigma_{j} \le \bar{\sigma}^P_2={ \max\left\{\bar{\sigma}^P_1, \ \frac{\gamma_2\bar{\rho}}{2} + \gamma_2\kappa_\theta+\gamma_2\eta+2\gamma_2, \ \gamma_2 L + \gamma_2\epsilon_0+ \gamma_2\bar\rho + 3\gamma_2 \kappa_\theta + 2\gamma_2\eta \right\}},
\end{equation}
where $\bar{\sigma}^P_1$ is responsible for an upper bound of $\sigma_{0}$. In addition, it follows from the construction of Algorithm \ref{Algorithm:SAARC} that $\sigma_{\min}\leq\sigma_{j}$ for all iterations, and $\gamma_1\sigma_{j}\leq\sigma_{j+1}$ for all unsuccessful iterations. Therefore, we have
\begin{equation*}
\frac{\bar{\sigma}^P_2}{\sigma_{\min}} \geq \frac{\sigma_{T_1+T_2}}{\sigma_{T_1}} = \prod_{j\in\SucCal} \frac{\sigma_{j+1}}{\sigma_{j}} \cdot \prod_{j\notin\SucCal} \frac{\sigma_{j+1}}{\sigma_{j}} \geq \gamma_1^{T_2-|\SucCal|}\left(\frac{\sigma_{\min}}{\bar{\sigma}^P_2}\right)^{|\SucCal|},
\end{equation*}
hence
\begin{equation*}
|\SucCal| \le T_2\leq |\SucCal|+\frac{\left(|\SucCal|+1\right)}{\log\gamma_1} \log\left(\frac{\bar{\sigma}^P_2}{\sigma_{\min}}\right) \leq \left(1+ \frac{2}{\log\gamma_1} \log\left(\frac{\bar{\sigma}^P_2}{\sigma_{\min}}\right)\right)|\SucCal|.
\end{equation*}
This completes the proof of Lemma \ref{Lemma:SAARC-T2}. {\hfill $\Box$ \vskip 0.4cm}

We present~\citet[Lemma~3.3 and~3.4]{Jiang-2020-Unified}, which are important to the subsequent analysis. 
\begin{lemma}\label{Lemma:General-Second-Order}
For any $\s\in\br^d$ and $\g\in\br^d$, it holds that
\begin{equation*}
\s^\top\g + \frac{1}{3}\sigma\|\s\|^3 \geq -\frac{2}{3\sqrt{\sigma}}\|\g\|^{\frac{3}{2}}.
\end{equation*}
\end{lemma}
\begin{lemma}\label{Lemma:SAARC-T3-P1}
Letting $\z_l = \argmin_{\z\in\br^d} \psi_l(\z)$, we have $\psi_l(\z) - \psi_l(\z_l) \geq (\varsigma_l/12)\|\z-\z_l\|^3$.
\end{lemma}
The following lemma is useful to bound the total number of successfully updating $\varsigma>0$.
\begin{lemma}\label{Lemma:SAARC-T3-P2}
Suppose in each iteration $j$ of Algorithm \ref{Alg:ASAS}, we have $\|\nabla^2 f(\x_j) - H(\x_j)\| \leq \epsilon_j$ for any $0 \leq j \leq T_2$. Then we have
\begin{equation*}
\|\nabla f(\x_{j+1}){\|} \leq {\left(0.5\bar{\rho}+ 2\kappa_\theta + L +{ \epsilon_0}+ 2\bar{\sigma}^P_2+2\right)}\|\s_j\|^2,
\end{equation*}
where $\kappa_\theta \in (0,1)$ is used in Condition~\ref{Cond:Approx_Subprob_SAARC}.
\end{lemma}
\begin{proof}
Note that $\nabla_{\s} m(\s_j;\x_j,\sigma_j) := \nabla f(\x_j) + H(\x_j)\s_j + \sigma_j\|\s_j\|\cdot\s_j $. Then we have
\begin{eqnarray*}
& & \|\nabla f(\x_{j+1})\| \\
& \leq & \|\nabla f(\x_{j+1}) - \nabla f(\x_j)- \nabla^2 f(\x_j)\s_j| + \|\nabla^2 f(\x_j) - H(\x_j)\|\|\s_j\| + \sigma_j\|\s_j\|^2 + \|\nabla_{\s} m(\s_j;\x_j,\sigma_j)\| \\
& \leq & \left\|\int_0^1(\nabla^2 f(\x_j+\tau \s_j)- \nabla^2 f(\x_j)) \s_j d\tau\right\| + 2 \epsilon_j\|\s_j\| + \sigma_j\|\s_j\|^2 + \kappa_\theta\|\s_j\|^2 \\
& \leq & \frac{\bar{\rho}}{2}\|\s_j\|^2 + 2 \epsilon_j \|\s_j\| + \bar\sigma_2\|\s_j\|^2 + \kappa_\theta \cdot \|\s_j\|^2 ,
\end{eqnarray*}
where the second inequality holds true due to Condition~\ref{Cond:Approx_Subprob_SAARC}, and the last inequality follows from Assumption \ref{Assumption-Objective-Gradient-Hessian} and \eqref{bound-sigma-j}. The subsequent analysis is conducted according to the value of $\|\s_j\|$ in two cases.
\begin{enumerate}
\item When $\|\s_j\|\geq 1$, we have
\begin{equation*}
\|\nabla f(\x_{j+1})\| \leq \left(\frac{\bar{\rho}}{2} + 2 \epsilon_j + \bar{\sigma}^P_2 + \kappa_\theta\right)\|\s_j\|^2,
\end{equation*}
which combined with $\epsilon_j \leq \epsilon_0 \le 1$ implies that
\begin{equation*}
\|\nabla f(\x_{j+1})\| \leq \left(\frac{\bar{\rho}}{2} + \bar{\sigma}^P_2 + \kappa_\theta + 2 \right)\|\s_j\|^2.
\end{equation*}
\item When $\|\s_j\|<1$, {recall in Algorithm \ref{Alg:ASAS}} that
\begin{equation*}
\epsilon_j = \min\left\{\frac{\left\|\nabla f({\y_l}) \right\|}{4}, \ {  \epsilon_{0}}\right\} \leq \frac{\left\|\nabla f({\y_l})  \right\|}{4},
\end{equation*}
which combined with the first inequality at the beginning of the proof implies that
\begin{eqnarray*}
\|\nabla f(\x_{j+1})\| & \leq & \left(\frac{\bar{\rho}}{2} + \kappa_\theta + \bar{\sigma}^P_2\right)\|\s_j\|^2 + \frac{\|\nabla f({\y_l})\|\|\s_j\|}{2} \\
& \leq & \left(\frac{\bar{\rho}}{2} + \kappa_\theta  + \bar{\sigma}^P_2\right)\|\s_j\|^2 +  {\left(L+{ \epsilon_0} + \kappa_\theta  + \sigma_j \right)\|\s_j\|^2}\\
& \leq & \left(\frac{\bar{\rho}}{2}+ 2\kappa_\theta + L +{ \epsilon_0} + 2\bar{\sigma}^P_2\right)\| \s_j\|^2.
\end{eqnarray*}
\end{enumerate}
In summary, we have
\begin{equation*}
\|\nabla f(\x_{j+1})\| \leq \left(\frac{\bar{\rho}}{2}+ 2\kappa_\theta + L+{ \epsilon_0}  + 2\bar{\sigma}^P_2+2\right)\|\s_j\|^2.
\end{equation*}
\end{proof}
\paragraph{Proof of Lemma \ref{Lemma:SAARC-T3}:} We are now ready to provide an upper bound of $T_3$. When $l=0$, it trivially holds true that $\psi_l(\z_l) \geq (1/6)l(l+1)(l+2)f(\bar{\x}_l)$ since $\psi_0(\z)=f(\bar{\x}_0)$. It suffices to establish the general case when $\varsigma_l \geq 8\eta^{-2}(0.5\bar{\rho} + 2\kappa_\theta + L + 2\bar{\sigma}^P_2+1)^3$ by mathematical induction. Without loss of generality, we assume \eqref{Inequality:induction} holds true for some $l - 1 \ge 1$. Then, it follows from Lemma \ref{Lemma:SAARC-T3-P1}, and the construction of $\psi_l(\z)$ that
\begin{equation*}
\psi_{l-1}(\z)\geq\psi_{l-1}(\z_{l-1})+\frac{1}{12}\varsigma_{l-1}\|\z-\z_{l-1}\|^3 \geq \frac{(l-1)l(l+1)}{6} f(\bar{\x}_{l-1}) + \frac{1}{12}\varsigma_{l-1}\|\z-\z_{l-1}\|^3.
\end{equation*}
As a result, we have
{\small \begin{eqnarray*}
& & \psi_l(\z_l) \\ 
& = & \min_{\z\in\br^d} \ \left\{\psi_{l-1}(\z) + \frac{l(l+1)}{2}[ f(\bar{\x}_l)+(\z-\bar{\x}_l)^\top \nabla f(\bar{\x}_l)] + \frac{1}{6}(\varsigma_l - \varsigma_{l-1})\|\z-\bar{\x}_0\|^3  \right\} \\
& \geq & \min_{\z\in\br^d} \ \left\{\frac{(l-1)l(l+1)}{6} f(\bar{\x}_{l-1}) + \frac{\varsigma_{l-1}}{12}\left\|\z-\z_{l-1}\right\|^3 + \frac{l(l+1)}{2}\left[ f(\bar{\x}_l)+\left(\z-\bar{\x}_l\right)^\top \nabla f(\bar{\x}_l) \right]  \right\} \\
& \geq & \min_{\z\in\br^d} \ \left\{\frac{(l-1)l(l+1)}{6} [f(\bar{\x}_l)+(\bar{\x}_{l-1}-\bar{\x}_l)^\top\nabla f(\bar{\x}_l)] + \frac{\varsigma_{l-1}}{12}\|\z- \z_{l-1}\|^3 + \frac{l(l+1)}{2}[f(\bar{\x}_l)+(\z -\bar{\x}_l)^\top\nabla f(\bar{\x}_l)]\right\} \\
& = & \frac{l(l+1)(l+2)}{6} f(\bar{\x}_l) + \min_{\z\in\br^d} \left\{\frac{(l-1)l(l+1)}{6} \left(\bar{\x}_{l-1}-\bar{\x}_l\right)^\top\nabla f(\bar{\x}_l)+ \frac{\varsigma_{l-1}}{12}\|\z-\z_{l-1}\|^3 + \frac{l(l+1)}{2}(\z-\bar{\x}_l)^\top  \nabla f(\bar{\x}_l)\right\},
\end{eqnarray*}}
where the first inequality follows from $\varsigma_l \ge \varsigma_{l-1}$.
By the construction of $\y_{l-1}$, we have
\begin{eqnarray*}
\frac{(l-1)l(l+1)}{6}\bar{\x}_{l-1} & = & \frac{l(l+1)(l+2)}{6}\cdot\frac{l-1}{l+2}\bar{\x}_{l-1} \\
& = & \frac{l(l+1)(l+2)}{6}\left(\y_{l-1} - \frac{3}{l+2}\z_{l-1}\right) \\
& = & \frac{l(l+1)(l+2)}{6}\y_{l-1} - \frac{l(l+1)}{2}\z_{l-1}.
\end{eqnarray*}
Combining the above two formulas yields
\begin{eqnarray*}
\psi_l(\z_l) & \geq & \frac{l(l+1)(l+2)}{6} f(\bar{\x}_l) + \min_{\z\in\br^d} \ \left\{\frac{l(l+1)(l+2)}{6}\left(\y_{l-1}-\bar{\x}_l\right)^\top\nabla f(\bar{\x}_l) \right. \\
& & \left. + \frac{\varsigma_{l-1}}{12}\|\z-\z_{l-1}\|^3 + \frac{l(l+1)}{2}(\z-\z_{l-1})^\top\nabla f(\bar{\x}_l) \right\}.
\end{eqnarray*}
By the criterion of successful iteration in Algorithm \ref{Alg:ASAS} and Lemma \ref{Lemma:SAARC-T3-P2}, we have
\begin{equation*}
(\y_{l-1}-\bar{\x}_l)^\top\nabla f(\bar{\x}_l) = -\s_j^\top\nabla f(\bar \x_l) \geq \eta\| \s_j\|^3 \geq \eta \left(\frac{\|\nabla f(\bar{\x}_l)\|}{0.5\bar{\rho}+ 2\kappa_\theta + L +{ \epsilon_0} + 2\bar{\sigma}^P_2+2 }\right)^{\frac{3}{2}},
\end{equation*}
where the $l$-th successful iteration count refers to the $j$-th iteration count. Hence, it suffices to establish
\begin{equation*}
\frac{l(l+1)(l+2)\eta}{6} \left(\frac{\|\nabla f(\bar{\x}_l)\|}{\frac{\bar{\rho}}{2}+ 2\kappa_\theta + L+{ \epsilon_0} + 2\bar{\sigma}^P_2+2}\right)^{\frac{3}{2}} + \frac{\varsigma_{l-1}}{12}\|\z-\z_{l-1}\|^3 + \frac{l(l+1)}{2}(\z-\z_{l-1})^\top  \nabla f(\bar{\x}_l) \geq 0.
\end{equation*}
Using Lemma \ref{Lemma:General-Second-Order} and setting $\g = 0.5l(l+1)\nabla f(\bar{\x}_l)$, $\s=\z-\z_l$, and $\sigma=\varsigma_{l-1}/4$, the above is implied by
\begin{equation}\label{Cubic-Proof-last-Inequality}
\frac{l(l+1)(l+2)\eta}{6}\left(\frac{1}{0.5\bar{\rho} + 2\kappa_\theta + L +{ \epsilon_0} + 2\bar{\sigma}^P_2+2}\right)^{\frac{3}{2}}\geq\frac{4}{3\sqrt{\varsigma_{l-1}}}\left(\frac{l(l+1)}{2}\right)^{\frac{3}{2}}.
\end{equation}
Therefore, the conclusion follows if
\begin{equation*}
\varsigma_{l-1} \geq {8\eta^{-2}(0.5\bar{\rho} + 2\kappa_\theta + L +{ \epsilon_0} + 2\bar{\sigma}^P_2+2)^3}.
\end{equation*}
This completes the proof.

\section{Proofs in Section \ref{Section:worst-case}}\label{Sec: Worst-Proofs}
We prove Lemma \ref{Lemma:W-AARC-T1} and Lemma \ref{Lemma:W-AARC-T2}, which describe the relation between the total iteration numbers in Algorithm \ref{Alg:SSAS} and the amount of successful iterations $|\SucCal|$ in Algorithm \ref{Alg:ASAS}.
\paragraph{Proof of Lemma \ref{Lemma:W-AARC-T1}: } According to Condition~\ref{Cond:Approx_Subprob_SAARC}, it holds that
\begin{eqnarray*}
\kappa_\theta\|\nabla f(\x_i)\| & \geq & \|\nabla m(\s_i; \x_i, \sigma_i))\| \ = \ \|\nabla f(\x_i) + H(\x_i) \s_i + \sigma_i\|\s_i\|\cdot\s_i\| \nonumber \\
&  {\eqref{Property: matr-H} \above 0pt \geq } & \|\nabla f(\x_i)\| - (L+ \epsilon_0) \|\s_i\| - \sigma_i\|\s_i\|^2,
\end{eqnarray*}
which implies that
\begin{equation*}
\sigma_i\|\s_i\|^2 + (L+\epsilon_0)\|\s_i\|  - (1-\kappa_\theta) \sqrt{\epsilon} \geq 0,
\end{equation*}
and hence
\begin{equation}\label{W-lowerbound-si}
\|\s_i\| \geq \frac{-(L + \epsilon_0)+\sqrt{(L+\epsilon_0)^2+4\sigma_i\sqrt{\epsilon} (1-\kappa_\theta)}}{2\sigma_i}.
\end{equation}
Moreover, we have that
\begin{eqnarray}\label{Equality-Second-Order-AARC}
f(\x_i + \s_i) & = & f(\x_i) + \s_i^\top \nabla f(\x_i) + \int_{0}^{1} \left[\nabla f(\x_i + \tau \s_i) -  \nabla f(\x_i)\right] \s_i \; d \tau \nonumber \\
& {\eqref{Def:Lipschitz-Gradient} \above 0pt  \leq } & f(\x_i) + \s_i^\top \nabla f(\x_i) + \frac{L}{2}\|\s_i\|^2   \nonumber \\
& = & m(\s_i; \x_i, \sigma_i)) + \frac{L}{2}  \| \s_i \|^2 - \frac{1}{2}\s_i^\top H(\x_i) \s_i - \frac{\sigma_i}{3}\|\s_i\|^3 \nonumber \\
& {\eqref{Property: matr-H} \above 0pt  \leq }  & m(\s_i; \x_i, \sigma_i)) + \left(\frac{L}{\|\s_i\|} - \frac{\sigma_i}{3}\right)\|\s_i\|^3.
\end{eqnarray}
Combining {\eqref{W-lowerbound-si}} and \eqref{Equality-Second-Order-AARC}
 yields the following relation
\begin{equation*}
\frac{2\sigma_i L}{-(L + \epsilon_0)+\sqrt{(L+\epsilon_0)^2+4\sigma_i\sqrt{\epsilon} (1-\kappa_\theta)} } - \frac{\sigma_i}{3} \leq 0 \Longrightarrow f(\x_i+\s_i)\leq m(\s_i; \x_i, \sigma_i)).
\end{equation*}
Note that the left hand side inequality is equivalent to
\begin{equation*}
\frac{(L + \epsilon_0)+\sqrt{(L+\epsilon_0)^2+4\sigma_i\sqrt{\epsilon} (1-\kappa_\theta)} }{2\sqrt{\epsilon} (1-\kappa_\theta) } - \frac{\sigma_i}{3 L} \leq 0 ,
\end{equation*}
which is implied by $\sigma_i \geq \frac{3L(4L + \epsilon_0)}{(1-\kappa_\theta)\sqrt{\epsilon}}$.
In summary, we have concluded that
\begin{equation*}\label{sigma-imply2}
\sigma_i \geq \frac{3L(4L + \epsilon_0)}{(1-\kappa_\theta)\sqrt{\epsilon}} \Longrightarrow f(\x_i+\s_i)\leq m(\s_i; \x_i, \sigma_i)).
\end{equation*}
The remaining proof is similar to the argument below \eqref{sigma-imply} in Lemma \ref{Lemma:SAARC-T1}. 

\paragraph{Proof of Lemma \ref{Lemma:W-AARC-T2}:}
We have
\begin{eqnarray*}
\s_{j}^\top\nabla f(\y_l+\s_{j}) & = & \s_{j}^\top\left[\nabla f(\y_l + \s_{j}) - \nabla f(\y_l) \right] + \s_{j}^\top\left[ \nabla f(\y_l) + H(\y_l)\s_{j}\right] {-} \s_{j}^\top H(\y_l)\s_{j} \\
& \leq & \|\nabla f(\y_l+\s_{j}) - \nabla f(\y_l)\|\|\s_j\| + \s_j^\top[\nabla m(\y_l, \s_{j}, \sigma_{j}) - \sigma_j\|\s_j\|\s_j] {+} \|H(\y_l)\| \| \s_j\|^2 \\
& {\eqref{Def:Lipschitz-Gradient} \eqref{Eqn:Approx_Subprob_SAARC} \eqref{Property: matr-H}\above 0pt \leq }  & L\|\s_j\|^2  + \kappa_\theta\|\s_j\|^3 - \sigma_j\| \s_j\|^3 + (L + \epsilon_0)\| \s_j\|^2 \\
& = & \left(\frac{2L + \epsilon_0}{\|\s_j\|} + \kappa_\theta - \sigma_{j} \right)\|\s_{j}\|^3.
\end{eqnarray*}
A similar argument of \eqref{W-lowerbound-si} implies that
\begin{equation*}
\|\s_j\| \geq \frac{-(L + \epsilon_0)+\sqrt{(L+\epsilon_0)^2+4\sigma_j\sqrt{\epsilon} (1-\kappa_\theta)}}{2\sigma_j}.
\end{equation*}
Now combining the two inequalities above yields the following relation.
\begin{equation*}
\frac{L+\epsilon_0 +\sqrt{(L+\epsilon_0)^2+4\sqrt{\epsilon}\sigma_{j}(1-\kappa_\theta)}}{2\sqrt{\epsilon}(1-\kappa_\theta)} - \frac{\sigma_{j} - \kappa_\theta - \eta}{2L + \epsilon_0} \leq 0 \Longrightarrow -\frac{\s_{j}^\top\nabla f(\y_l+\s_{j})}{\|\s_{j}\|^3}\geq\eta.
\end{equation*}
A straight forward calculation shows that the inequality on the left hand side is implied by
\begin{equation*}\label{upperbound-epsiloni}
\sigma_{j} \geq \frac{(3L+2\epsilon_0)(2L+\epsilon_0)+2\sqrt{\epsilon}(1-\kappa_\theta)(\kappa_\theta+\eta)+(2L+\epsilon_0)\sqrt{(3L+2\epsilon_0)^2+\sqrt{\epsilon}(1-\kappa_\theta)(\kappa_\theta+\eta)}}{2 \sqrt{\epsilon}(1-\kappa_\theta)} .
\end{equation*}
The remaining proof is similar to the argument in Lemma \ref{Lemma:SAARC-T2}. 
\begin{lemma}\label{Lemma:W-AARC-T3-P2}  
Suppose in each iteration $j$ of Algorithm \ref{Alg:ASAS}, we have $\|\nabla f(\x_j)\|^2 > \epsilon$ for any $0 \leq j \leq T_2$. Then we have
\begin{equation*}
\|\nabla f(\x_{j+1})\| \leq \left((2L+\epsilon_0) \cdot \frac{(L+\epsilon_0) +\sqrt{(L+\epsilon_0)^2+4{\bar \sigma}^W_{2}\sqrt{\epsilon}(1-\kappa_\theta)}}{2\sqrt{\epsilon}(1-\kappa_\theta)} + \bar{\sigma}^W_2 + \kappa_\theta \right) \|\s_j\|^2
\end{equation*}
where $\kappa_\theta\in(0,1)$ is used in Condition~\ref{Cond:Approx_Subprob_SAARC}.
\end{lemma}
\begin{proof}
Recalling $\nabla_{\s} m(\s_j; \x_j, \sigma_j)=\nabla f(\x_j) + H(\x_j)\s_j + \sigma_j\|\s_j\|\cdot\s_j$, we have
\begin{eqnarray*}
\|\nabla f(\x_{j+1})\| & \leq & \left\| \nabla f(\x_j+\s_j) - \nabla_{\s} m(\s_j; \x_j, \sigma_j)\right\| + \left\|\nabla_{\s} m(\s_j; \x_j, \sigma_{j})\right\| \\
&  {\eqref{Eqn:Approx_Subprob_SAARC} \above 0pt \leq } & \|\nabla f(\x_j+\s_j) - \nabla_{\s} m(\s_j; \x_j, \sigma_j)\| + \kappa_{\theta}\|\s_j\|^2 \\
& \leq & \|\nabla f(\x_j+\s_j) - \nabla f({\x_j})\| + \| H(\x_j)\|\|\s_j\| + \sigma_j\| \s_j\|^2 + \kappa_{\theta}\| \s_j \|^2 \\
& {\eqref{Def:Lipschitz-Gradient} \eqref{Property: matr-H}\above 0pt \leq } & L\|\s_j\| + (L + \epsilon_0)\|\s_j\| + \sigma_j\|\s_j\|^2 + \kappa_{\theta}\| \s_j \|^2\\
& = & \left(\frac{2L + \epsilon_0}{\|\s_j\|}+ \bar \sigma^W_2+ \kappa_{\theta}\right)\| \s_j\|^2.
\end{eqnarray*}
A similar argument of \eqref{W-lowerbound-si} implies that
\begin{equation*}
\|\s_j\| \geq \frac{-(L+\epsilon_0)+\sqrt{(L+\epsilon_0)^2+4 {\sigma}_{j}\sqrt{\epsilon}(1-\kappa_\theta)}}{2 {\sigma}_j}.
\end{equation*}
Therefore, we conclude that
\begin{equation*}
\|\nabla f(\x_{j+1})\| \leq \left((2L+\epsilon_0) \cdot \frac{(L+\epsilon_0) +\sqrt{(L+\epsilon_0)^2+4{\bar \sigma}^W_{2}\sqrt{\epsilon}(1-\kappa_\theta)}}{2\sqrt{\epsilon}(1-\kappa_\theta)}  + \bar{\sigma}^W_2 + \kappa_\theta \right)\|\s_j\|^2.
\end{equation*}
\end{proof}
We are now ready to provide an upper bound of $T_3$.
\paragraph{Proof of Lemma \ref{Lemma:W-AARC-T3}:} The proof is almost the same as that of Lemma \ref{Lemma:SAARC-T3} by mathematical induction. The only difference is the estimation of
\begin{eqnarray*}
& & (\y_{l-1}-\bar{\x}_l)^\top\nabla f(\bar{\x}_l)  \\
& \geq & \eta \left(   \frac{2\sqrt{\epsilon}(1-\kappa_\theta) }{\left(2L+\epsilon_0\right)\left( (L+\epsilon_0) +\sqrt{(L+\epsilon_0)^2+4{\bar \sigma}^W_{2}\sqrt{\epsilon}(1-\kappa_\theta)} \right)+2\sqrt{\epsilon}(1-\kappa_\theta)\left( \bar{\sigma}^W_2 + \kappa_\theta  \right)}  \right)^{\frac{3}{2}}\left\| \nabla f(\bar{\x}_l)\right\|^{\frac{3}{2}},
\end{eqnarray*}
{which is due to Lemma \ref{Lemma:W-AARC-T3-P2}.}
By adapting the proof of Lemma \ref{Lemma:SAARC-T3} with such estimation, we achieve the desired result.

\end{document}